\documentclass{amsart}%
\usepackage{amssymb}
\usepackage{amsmath}
\usepackage{amsfonts}
\usepackage[margin=1.45in]{geometry}
\usepackage[all,cmtip]{xy}
\usepackage{hyperref}%
\usepackage{graphicx}

\newtheorem{theorem}{Theorem}
\newtheorem*{thm_announce}{Theorem}
\theoremstyle{plain}
\newtheorem*{acknowledgement}{Acknowledgement}

\newtheorem{conjecture}{Conjecture}
\newtheorem{corollary}[theorem]{Corollary}

\newtheorem*{def_announce}{Definition}
\newtheorem{definition}[theorem]{Definition}

\newtheorem{lemma}[theorem]{Lemma}

\newtheorem{problem}{Question}
\newtheorem{proposition}[theorem]{Proposition}
\theoremstyle{remark}
\newtheorem{remark}[theorem]{Remark}
\newtheorem{dream}[theorem]{Dream}

\newtheorem{example}[theorem]{Example}

\numberwithin{equation}{section}
\numberwithin{theorem}{section}
\begin{document}
\title[Automorphic side of the Artin map]{On the automorphic side of the $K$-theoretic\linebreak Artin symbol}
\author{Peter Arndt}
\address{Mathematical Institute of the Heinrich-Heine-Universit\"{a}t D\"{u}sseldorf,
Universit\"{a}tsstr. 1, 40225 D\"{u}sseldorf, Germany}
\email{peter.arndt@uni-duesseldorf.de}
\author{Oliver Braunling}
\address{Freiburg Institute for Advanced Studies (FRIAS), University of Freiburg, 79104
Freiburg im\ Breisgau, Germany}
\email{oliver.braeunling@math.uni-freiburg.de}
\thanks{The second author was supported by DFG GK1821 \textquotedblleft Cohomological
Methods in Geometry\textquotedblright\ and a Junior Fellowship at the Freiburg
Institute for Advanced Studies (FRIAS)}

\begin{abstract}
Clausen has constructed a homotopical enrichment of the Artin reciprocity
symbol in class field theory. On the Galois side, Selmer $K$-homology replaces
the abelianized Galois group, while on the automorphic side the $K$-theory of
locally compact vector spaces replaces classical idelic objects. We supply
proofs for some predictions of Clausen regarding the automorphic side.

\end{abstract}
\maketitle

This article proves certain predictions of Clausen's programme to enrich class
field theory homotopically \cite{clausen}. We recall the motivation: For a
field $F$, let $F^{\operatorname*{ab}}$ denote its maximal abelian extension.
Classical class field theories can all be phrased in terms of a reciprocity
map. Depending on what type of field we deal with, these take the following
forms:%
\begin{equation}%
\begin{array}
[c]{rcl}%
\mathbb{Z}\longrightarrow\operatorname*{Gal}(F^{\operatorname*{ab}}/F) &
\qquad & \text{if }F\text{ is a finite field,}\\
F^{\times}\longrightarrow\operatorname*{Gal}(F^{\operatorname*{ab}}/F) &
\qquad & \text{if }F\text{ is a local field,}\\
\mathbb{A}^{\times}/F^{\times}\longrightarrow\operatorname*{Gal}%
(F^{\operatorname*{ab}}/F) & \qquad & \text{if }F\text{ is a number field,}%
\end{array}
\label{lint_1}%
\end{equation}
where $\mathbb{A}$ denotes the ad\`{e}les. The left hand side is called the
\textit{automorphic side}, while the right hand side is the \textit{Galois
side}. Two natural questions: (A)~Is~there a uniform explanation/definition
for what group should go on the automorphic side for a more general field $F$?
(B)~Is~there an enrichment of the reciprocity map which might see more on the
Galois side than just the abelianization of the absolute Galois group?

Clausen's recent programme \cite{clausen} proposes and proves that there
should be a uniform (and moreover homotopically enriched) formulation for all
these reciprocity maps on the same footing:\ He constructs a map%
\[%
\begin{array}
[c]{rcl}%
K(\mathsf{LCA}_{F})\longrightarrow\left.  dK^{\operatorname*{Sel}}\right.
(F) &  & \text{for any }F\text{ as above,}%
\end{array}
\]
where the left hand side denotes the non-connective algebraic $K$-theory
spectrum of the category $\mathsf{LCA}_{F}$ of locally compact $F$-vector
spaces, and $\left.  dK^{\operatorname*{Sel}}\right.  $ denotes Clausen's
so-called `Selmer $K$-homology spectrum' of the category of finite-dimensional
$F$-vector spaces. His reciprocity map is a functorial morphism between these
spectra. Such a morphism induces a morphism between the respective homotopy
groups, and in degree one it specializes to%
\[
\pi_{1}K(\mathsf{LCA}_{F})\longrightarrow\pi_{1}\left.
dK^{\operatorname*{Sel}}\right.  (F)\text{.}%
\]
The vision is that upon unravelling these $\pi_{1}$-groups, this map recovers
the respective map of Equation \ref{lint_1}, uniformly for all three types of
fields $F$. This gives great candidates to answer the above questions:
(A)~In~general the $K$-theory of some category of locally compact modules
$\mathsf{LCA}_{F}$ should uniformly give the correct automorphic side, and
this might conjecturally hold for much broader types of fields $F$, or
schemes...; (B)~The~abelianized Galois group gets replaced by a richer object
which knows more than just the abelianization of the Galois group, but also
further Galois cohomological data.\medskip

Clausen constructs this homotopically enriched Artin symbol. His first task
was then to prove that his map recovers classical class field theory: In all
three cases above, the Galois side matches perfectly, i.e. it is uniformly
$\pi_{1}\left.  dK^{\operatorname*{Sel}}\right.  (F)\cong$
$\operatorname*{Gal}(F^{\operatorname*{ab}}/F)$.\ Regarding the automorphic
side, he shows that there are canonical maps $\psi$ such that the compositions%
\begin{equation}%
\begin{array}
[c]{rcl}%
\mathbb{Z}\overset{\psi}{\longrightarrow}\pi_{1}K(\mathsf{LCA}_{F}%
)\longrightarrow\operatorname*{Gal}(F^{\operatorname*{ab}}/F) & \qquad &
\text{if }F\text{ is a finite field,}\\
F^{\times}\overset{\psi}{\longrightarrow}\pi_{1}K(\mathsf{LCA}_{F}%
)\longrightarrow\operatorname*{Gal}(F^{\operatorname*{ab}}/F) & \qquad &
\text{if }F\text{ is a local field,}\\
\mathbb{A}^{\times}/F^{\times}\overset{\psi}{\longrightarrow}\pi
_{1}K(\mathsf{LCA}_{F})\longrightarrow\operatorname*{Gal}%
(F^{\operatorname*{ab}}/F) & \qquad & \text{if }F\text{ is a number field,}%
\end{array}
\label{lint_2}%
\end{equation}
indeed give exactly the classical reciprocity maps of Equation \ref{lint_1}.
He then predicts that the maps $\psi$ should be isomorphisms. The main
motivation for this paper was to prove this:

\begin{thm_announce}
For finite fields, $p$-adic local fields, and number fields, the maps $\psi$
in Equation \ref{lint_2} are isomorphisms.\footnote{For the $p$-adic local
fields we must assume the scalar action of the field $F$ on the locally
compact vector space to be continuous in the $p$-adic topology.}
\end{thm_announce}

See Theorem \ref{thm_Main_Agreement}. While we originally only wanted to prove
this claim, we obtain this now as a rather tiny corollary of much stronger
results. Firstly, we give a complete description of the $K$-theory spectra on
the automorphic side in all three cases. With a minor modification it confirms
predictions by Clausen:

\begin{thm_announce}
Suppose $F$ is

\begin{enumerate}
\item a finite field, then $K(\mathsf{LCA}_{F})\overset{\sim}{\longrightarrow
}\Sigma K(F)$;

\item a $p$-adic local field, then $K(\mathsf{LCA}_{F,\operatorname*{top}%
})\overset{\sim}{\longrightarrow}K(F)$;

\item a number field, then there is a fiber sequence%
\begin{equation}
K(F)\longrightarrow K(\mathsf{LCA}_{F,ab})\longrightarrow K(\mathsf{LCA}%
_{F})\text{,}\label{lint_W1}%
\end{equation}
where $\mathsf{LCA}_{F,ab}$ is the category of adelic blocks (explained below).
\end{enumerate}

These results are not just valid for non-connective $K$-theory, but for any
localizing invariant $K:\operatorname*{Cat}_{\infty}^{\operatorname*{ex}%
}\rightarrow\mathsf{A}$ in the sense of \cite{MR3070515} with values in a
stable $\infty$-category $\mathsf{A}$.
\end{thm_announce}

So, for example, this also describes the Hochschild homology over $F$,
topological cyclic homology, etc. of $\mathsf{LCA}_{F}$. In the $p$-adic local
case the `$\operatorname*{top}$' subscript indicates that we demand that the
$F$-scalar action is continuous with respect to the $p$-adic topology. We
refer to Proposition \ref{prop_KLCA_for_finite_fields}, Corollary
\ref{cor_Main_KThy_LocalCase} and Theorem \ref{thm_MainKOfLCAF} for the
proofs. Above, $\mathsf{LCA}_{F,ab}$ denotes the category of \emph{adelic
blocks}, a concept which we introduce in this paper. It can be described as follows:

\begin{def_announce}
$\mathsf{LCA}_{F,ab}$ is the fully exact subcategory of $\mathsf{LCA}_{F}$
whose objects can be presented as a direct sum $A\oplus B$, where

\begin{enumerate}
\item the underlying additive group of $A$ is a finite-dimensional real vector
space, and

\item the underlying additive group of $B$ is topologically torsion.
\end{enumerate}

We call these objects \emph{adelic blocks}.
\end{def_announce}

Equivalently, the summand $B$ can be characterized as those modules which are
totally disconnected and their Pontryagin dual is also totally disconnected.
In the number field case, the $K$-theory of $\mathsf{LCA}_{F}$ thus reduces to
computing the $K$-theory of $\mathsf{LCA}_{F,ab}$. To handle the latter, we
show that the objects in this category all have the shape of restricted
products. Then, using that $K$-theory commutes with filtering colimits and
(and this is crucial!) with infinite products of categories, we prove the following:

\begin{thm_announce}
Let $F$ be a number field. Let $\widehat{F}_{P}$ (resp. $\widehat{\mathcal{O}%
}_{P}$) denote the local field of $F$ at a place $P$ (resp. its ring of
integers in the case of a finite place). For all integers $n$, there is a
canonical isomorphism of $K$-theory groups%
\[
K_{n}(\mathsf{LCA}_{F,ab})\cong\left\{  \left.  (\alpha_{P})_{P}\in\prod
_{P}K_{n}(\widehat{F}_{P})\right\vert
\begin{array}
[c]{l}%
\alpha_{P}\in K_{n}(\widehat{\mathcal{O}}_{P})\text{ for all but finitely}\\
\text{many among the finite places}%
\end{array}
\right\}  \text{.}%
\]
Here the product is taken over all places $P$, both finite and infinite, and
the condition makes sense since $K_{n}(\widehat{\mathcal{O}}_{P})\subseteq
K_{n}(\widehat{F}_{P})$ is naturally a subgroup.
\end{thm_announce}

See Theorem \ref{thm_MainStructureThmOnKLCAFab}. We actually get an analogous
description of the entire spectrum, but for the moment let us just deal with
the above formulation. For $K_{1}$, we immediately get that%
\[
K_{1}(\mathsf{LCA}_{F,ab})\cong\mathbb{A}^{\times}\text{,}%
\]
i.e. the $K_{1}$-group is isomorphic to the id\`{e}les. Sequence \ref{lint_W1}
then shows that $K_{1}(\mathsf{LCA}_{F})\cong\mathbb{A}^{\times}/F^{\times}$,
so $\pi_{1}K(\mathsf{LCA}_{F})$ is indeed the id\`{e}le class group. We should
note that there are a number of subtle issues hidden here: First of all, we
have%
\[
K(\mathsf{LCA}_{F,ab})\neq K(\mathbb{A})\text{,}%
\]
i.e. the $K$-theory spectrum of adelic blocks is\textit{\ truly different
}from the $K$-theory of the ring of ad\`{e}les of the number field, see
Example \ref{example_KAdifferentFromKAdelicBlocks}. Just to illustrate some
further subtleties, suppose $F=\mathbb{Q}$. For any cardinal $\omega$ the
union%
\[
\bigcup_{n\geq1}\frac{1}{p^{n}}\underset{\omega}{\prod}\mathbb{Z}_{p}%
\qquad\subseteq\qquad\underset{\omega}{\prod}\mathbb{Q}_{p}\text{,}%
\]
topologized such that $\prod_{\omega}\mathbb{Z}_{p}$ is a compact clopen, is
locally compact and algebraically a $\mathbb{Q}$-vector space. Yet, the scalar
action of $\mathbb{Q}$ is continuous if and only if $\omega<\infty$. So there
is a huge gap between the uniquely divisible groups in $\mathsf{LCA}%
_{\mathbb{Z}}$ and groups in $\mathsf{LCA}_{\mathbb{Q}}$. See Example
\ref{example_Robertson}. In a different direction, for any prime number $p$,
we may pick a cardinal $\omega_{p}$ and form the restricted product%
\[
\left.  \prod\nolimits^{\prime}\right.  (\mathbb{Q}_{p},\mathbb{Z}%
_{p})^{\omega_{p}}\text{,}%
\]
where the factor $(\mathbb{Q}_{p},\mathbb{Z}_{p})$ is featured with
$\omega_{p}$ copies. For any choices of $\omega_{p}$, this is an object of
$\mathsf{LCA}_{\mathbb{Z}}$, in particular it is always locally compact.
However, it is only an object of $\mathsf{LCA}_{\mathbb{Q}}$ if all cardinals
$\omega_{p}$ are finite. All objects which come from finitely generated
$\mathbb{A}$-modules, say with $g$ generators, have the uniform bound
$\omega_{p}\leq g$. This causes the discrepancy between $K(\mathbb{A}) $ and
$K(\mathsf{LCA}_{\mathbb{Q},ab})$, where no such uniform bound occurs. This is
not just a difference between the categories; it also affects the $K$-theory.
The resemblance of $K_{n}(\mathsf{LCA}_{F,ab})$ to the ad\`{e}les, i.e. `as if
$\omega_{p}=1$' for all $p$ then comes from the familiar rank stabilization
phenomenon in $K$-theory, i.e. in the case at hand that%
\[
\operatorname*{GL}\nolimits_{n}(\mathbb{Q}_{p})/[\operatorname*{GL}%
\nolimits_{n}(\mathbb{Q}_{p}),\operatorname*{GL}\nolimits_{n}(\mathbb{Q}%
_{p})]\cong\mathbb{Q}_{p}^{\times}%
\]
irrespective of the allowed ranks $n$.

It remains to prove that adelic blocks actually always need to have this form,
a result in the tradition of the works of Braconnier and Vilenkin on
topological torsion groups:

\begin{thm_announce}
Let $F$ be a number field and let $\mathcal{O}$ be its ring of integers. Every
object $G\in\mathsf{LCA}_{F,ab}$ is isomorphic to a restricted product%
\[
G\cong\bigoplus_{\sigma\in I}\mathbb{R}_{\sigma}\oplus\underset{P}{\left.
\prod\nolimits^{\prime}\right.  }(G^{P},G^{P}\cap C)\text{,}%
\]
where $I$ is a finite list of real and complex places, $P$ runs over all
finite places of $F$, $G^{P}$ denotes the topological $P$-torsion elements of
$G$, and $C$ is an arbitrary compact clopen $\mathcal{O}$-submodule of $G$. Moreover,

\begin{enumerate}
\item $G^{P}\cap C$ is a compact clopen $\mathcal{O}$-submodule of $G^{P}$;

\item each $G^{P}$ is a finite-dimensional $\widehat{F}_{P}$-vector space,
where the vector space structure is canonically given by $P$-adically
completing the $F$-vector space structure of $G^{P}$;

\item $G^{P}\cap C$ is a full rank $\widehat{\mathcal{O}}_{P}$-lattice inside
$G^{P}$;

\item $G$ is canonically a topological $\mathbb{A}$-module, where $\mathbb{A}
$ denotes the ad\`{e}les of $F$ as a locally compact ring.
\end{enumerate}
\end{thm_announce}

See Theorem \ref{thm_StrongBracVilenkin}. Understanding the adelic blocks
dominates the structure theory of the entire category $\mathsf{LCA}_{F}$:

\begin{thm_announce}
Let $G\in\mathsf{LCA}_{F}$ be arbitrary. Then there exists a direct sum
decomposition%
\[
G\cong V\oplus A\oplus D\text{,}%
\]
where $V$ is a compact $F$-vector space, $A$ is an adelic block, and $D$ is a
discrete $F$-vector space. Pontryagin duality exchanges the full subcategories
of compact vs. discrete $F$-vector spaces, while adelic blocks are closed
under Pontryagin duality. Even more strongly, all adelic blocks are self-dual,
i.e. $A\simeq A^{\vee}$ (non-canonically).
\end{thm_announce}

See Theorem \ref{thm_PrincipalStructureTheoremForLCAF}. Let us comment a
little further on the method of the paper: In a previous paper the variant for
the ring of integers, $\mathsf{LCA}_{\mathcal{O}}$, i.e. locally compact
$\mathcal{O}$-modules, was considered. The main result was a fiber sequence%
\[
K(\mathcal{O})\longrightarrow K(\mathbb{R})^{r}\oplus K(\mathbb{C}%
)^{s}\longrightarrow K(\mathsf{LCA}_{\mathcal{O}})\text{,}%
\]
where $r$ is the number of real places of $F$, and $s$ the number of complex
places, \cite{obloc}. This statement generalizes a theorem due to Clausen, who
established the case of the rational number field $\mathbb{Q}$, \cite{clausen}%
. We adapt the method of \cite{obloc}. \textit{Loc. cit.} most proofs were
based on decomposing a given locally compact $\mathcal{O}$-module $G$ as%
\[
V\oplus C\hookrightarrow G\twoheadrightarrow D
\]
with $V$ an $\mathcal{O}$-module whose underlying additive group is
topologically isomorphic to some real vector space, $C$ a compact
$\mathcal{O}$-module, $D$ discrete. However, this basic structural
decomposition is not available in $\mathsf{LCA}_{F}$ since $C$ or $D$ rarely
happen to be closed under scalar multiplication by all of $F$. A principal
technical ingredient of \textit{loc. cit.} was to show that the compactly
generated $\mathcal{O}$-modules form a left $s$-filtering subcategory of
$\mathsf{LCA}_{\mathcal{O}} $. In this paper, this will be replaced by the
category of \textquotedblleft direct sums of quasi-adelic blocks with a
finite-dimensional discrete $F$-vector space\textquotedblright. As in
\textit{loc. cit.}\ this category is quite precisely tailored to have the
relevant left $s$-filtering property. Incidentally, all objects in this
category admit a presentation as%
\[
G=\bigcup_{n\geq1}\frac{1}{n}C\text{,}%
\]
where $C\in\mathsf{LCA}_{\mathcal{O}}$ is a compactly generated $\mathcal{O}
$-module. So, there are strong parallels between these choices. Nonetheless,
practically all of the details change in comparison to \textit{loc. cit.} in
some way. For example, all compact modules in $\mathsf{LCA}_{F}$ are connected
and injective objects, quite unlike the compact $\mathcal{O}$-modules, where
we can find objects like $\mathbb{Z}_{p}$, which are both totally
disconnected, and clearly not an injective object since%
\[
\mathbb{Z}_{p}\hookrightarrow\mathbb{Q}_{p}\twoheadrightarrow\mathbb{Q}%
_{p}/\mathbb{Z}_{p}%
\]
is an exact sequence in $\mathsf{LCA}_{\mathbb{Z}}$. If $\mathbb{Z}_{p}$ was
injective, this sequence would have to split, which is false. Generally, all
objects like $\mathbb{R}$ or the $\mathbb{Q}_{p}$ happen to be more
interesting in $\mathsf{LCA}_{F}$. In $\mathsf{LCA}_{\mathcal{O}}$ they all
decompose into a compact and a discrete part, for $\mathbb{Q}_{p}$ as spelled
out above, and for $\mathbb{R}$ as in%
\[
\mathbb{Z}\hookrightarrow\mathbb{R}\twoheadrightarrow\mathbb{T}\text{,}%
\]
where $\mathbb{T}$ is the circle group (Note that the place of the compact
resp. discrete part is reversed from the non-archimedean cases). In
$\mathsf{LCA}_{F}$ none of these decompositions exist. This, in a way,
explains why we get the rather rich adelic $K$-theory part $K(\mathsf{LCA}%
_{F,ab})$ in the statement of our main theorem. Indeed, all local fields,
archimedean or non-archimedean, are irreducible objects in $\mathsf{LCA}_{F}$.

\begin{acknowledgement}
We, along with N. Hoffmann and M. Spitzweck, were at the University of
G\"{o}ttingen at the time when their paper \cite{MR2329311} was born. We were
fascinated right away by its underlying dreams, and this text was a wonderful
opportunity to return to this circle of ideas. We also thank Dustin Clausen,
Brad Drew and Matthias Wendt for taking a lot of time to discuss, and a number
of e-mails.
\end{acknowledgement}

\section{Preparations}

For us, a ring will always be a commutative unital ring. We write
\textquotedblleft list\textquotedblright\ with the meaning that repetitions
are allowed. Given a ring $R$, we tacitly equip it with the discrete topology
and we write $\mathsf{LCA}_{R}$ for the category of \emph{locally compact }%
$R$\emph{-modules}. That is, objects are $R$-modules $M$ along with the datum
of a locally compact topology on the underlying additive group $(M;+)$, and we
demand that for every $r\in R$ the scalar multiplication $M\overset{r\cdot
}{\rightarrow}M$ is continuous with respect to this topology. The morphisms of
$\mathsf{LCA}_{R}$ are continuous $R$-module homomorphisms.

Let $\mathbb{T}:=\mathbb{R}/\mathbb{Z}$ denote the circle as an object of
$\mathsf{LCA}_{\mathbb{Z}}$. We record an observation basically due to
Hoffmann and Spitzweck:

\begin{proposition}
\label{prop_LCARisQuasiAbelianExact}The category $\mathsf{LCA}_{R}$ is
quasi-abelian. In particular, it has a canonical exact category structure.
Moreover, there exists a canonical exact anti-equivalence of exact categories%
\[
(-)^{\vee}:\mathsf{LCA}_{R}^{op}\overset{\sim}{\longrightarrow}\mathsf{LCA}%
_{R}\text{,}\qquad\qquad G\longmapsto\operatorname*{Hom}\nolimits_{cts}%
(G,\mathbb{T})
\]
such that the natural reflexivity morphism $\eta_{G}:G\overset{\sim
}{\rightarrow}G^{\vee\vee}$ is an isomorphism. On the level of the underlying
additive groups, this agrees with ordinary Pontryagin duality.\ In particular,
it sends compact $R$-modules to discrete $R$-modules, and conversely. This
duality turns $\mathsf{LCA}_{R}$ into an exact category with duality in the
sense of \cite[Definition 2.1]{MR2600285}.
\end{proposition}

We refer to \cite{MR2329311} and \cite{obloc} for details. We also use the
shorthand $\mathsf{LCA}:=\mathsf{LCA}_{\mathbb{Z}}$, which is literally
equivalent to the category of locally compact abelian groups of
\cite{MR2329311}. In this case the duality is literally Pontryagin duality.

\begin{example}
We stress that the duality on $\mathsf{LCA}_{R}$ is very different from the
more conventional duality functor $G\mapsto\operatorname*{Hom}\nolimits_{R}%
(G,R)$. For example, $\mathbb{Z}^{\vee}=\mathbb{T}$, and unless $R$ is a
finite ring, $R$ itself (with the discrete topology) is never self-dual in
$\mathsf{LCA}_{R}$.
\end{example}

When we speak of exact categories, we use the conventions of B\"{u}hler's
survey \cite{MR2606234}. As part of this, morphisms denoted by an arrow
\textquotedblleft$\hookrightarrow$\textquotedblright\ are admissible monics,
whereas the arrow \textquotedblleft$\twoheadrightarrow$\textquotedblright%
\ indicates an admissible epic. Moreover, we shall make crucial use of the
following concepts due to Schlichting \cite{MR2079996}, \cite[\S 2.2]%
{MR3510209}:

\begin{definition}
[Schlichting]Suppose $\mathsf{D}$ denotes an exact category.

\begin{enumerate}
\item We call a full subcategory $\mathsf{C}\subset\mathsf{D}$ \emph{left
special} if for every object $Z\in\mathsf{C}$ and every admissible epic
$G\twoheadrightarrow Z$ in the category $\mathsf{D}$ one can find a
commutative diagram%
\[%
\xymatrix{
X \ar@{^{(}->}[r] \ar[d] & Y \ar@{->>}[r] \ar[d] & Z \ar[d]^{\operatorname
{id}} \\
F \ar@{^{(}->}[r] & G \ar@{->>}[r] & Z,
}%
\]
where the rows are exact, and the objects in the top row lie in $\mathsf{C}$.

\item We call a full subcategory $\mathsf{C}\subset\mathsf{D}$ \emph{left
filtering} if every morphism $G\rightarrow X$ in $\mathsf{D}$ such that
$G\in\mathsf{C}$ can be factored as%
\[%
\xymatrix{
G \ar[r] \ar[rd] & X \\
& Z, \ar@{^{(}->}[u]
}%
\]
where $Z\in\mathsf{C}$.
\end{enumerate}

We call $\mathsf{C}\subset\mathsf{D}$ \emph{left }$s$\emph{-filtering}, if it
is simultaneously left filtering and left special.
\end{definition}

The above definition does not use the original wording of \textit{loc. cit.},
but a convenient reformulation due to B\"{u}hler. A proof that both
definitions agree is given in \cite[Appendix A]{MR3510209}. The punchline
behind these definitions is the following basic principle: If $\mathsf{C}%
\subset\mathsf{D}$ is left $s$-filtering, then a quotient exact category
$\mathsf{D}/\mathsf{C}$ exists. See \cite[Proposition 1.16]{MR2079996}. On the
derived level quotients would exist far more generally, but we want to have
concrete exact category models for them.

\section{Basic structure results}

Let $F$ be a number field and $\mathcal{O}$ its ring of integers. Every
locally compact $F$-vector space is in particular a locally compact
$\mathcal{O}$-module.

\begin{lemma}
\label{lemma_LCAFtoLCAO}The natural functor $\mathsf{LCA}_{F}\rightarrow
\mathsf{LCA}_{\mathcal{O}}$ is fully faithful, exact and reflects exactness.
\end{lemma}

\begin{proof}
\textit{(Faithful)} The functor is faithful since maps are uniquely determined
by the map of underlying additive groups.\textit{\ (Full)} If $G,H\in
\mathsf{LCA}_{F}$ and $f:G\rightarrow H$ is an $\mathcal{O}$-module
homomorphism, we have $nf(\frac{1}{n}g)=f(g)$ for all $n\geq1$. We deduce
$f(\frac{1}{n}g)=\frac{1}{n}f(g)$ using the $F$-vector space structure of $H$,
and since all elements of $F$ can be written as $\alpha/n$ with $\alpha
\in\mathcal{O}$, $n\geq1$, we obtain that $f$ is $F$-linear. Thus, the functor
is full.\textit{\ (Exactness) }Given any exact sequence of
plain\ (untopologized) $F$-vector spaces, it is by definition of the exact
structure on $\mathsf{LCA}_{F}$ exact iff the underlying sequence of additive
groups is exact in $\mathsf{LCA}$, so exactness and reflection of exactness
are tautological.
\end{proof}

We recall from \cite{obloc} that for every real or complex embedding
$\nu:F\rightarrow\mathbb{R}$ (resp. $\mathbb{C}$) we write $\mathbb{R}_{\nu}$
for a copy of $\mathbb{R}$ (resp. $\mathbb{C}$), turned into a topological $F
$-vector space by the scalar multiplication%
\[
\alpha\cdot x:=\sigma(\alpha)x\qquad\text{for}\qquad\alpha\in F\text{, }%
x\in\mathbb{R}_{\nu}\text{.}%
\]
Thus, $\mathbb{R}_{\nu}\in\mathsf{LCA}_{F}$. Note that if $\nu,\overline{\nu}$
are complex conjugate embeddings, then complex conjugation induces an
isomorphism $\mathbb{R}_{\nu}\cong\mathbb{R}_{\overline{\nu}}$. Thus, we may
unambiguously speak of $\mathbb{R}_{\nu}$ even if $\nu$ is given plainly as a
real or complex place.

\begin{lemma}
\label{Lem_InjProjInLCAF}Regarding injective and projective objects, we observe:

\begin{enumerate}
\item Let $\nu$ be a real or a complex place of $F$. Then the object
$\mathbb{R}_{\nu}$ is both projective and injective in $\mathsf{LCA}_{F}$.

\item All discrete $F$-vector spaces in $\mathsf{LCA}_{F}$ are projective.

\item All compact $F$-vector spaces in $\mathsf{LCA}_{F}$ are injective.
\end{enumerate}
\end{lemma}

\begin{proof}
(1) This is quite immediate: Let us check injectivity. As $\mathbb{R}_{\nu}$
is injective in $\mathsf{LCA}_{\mathcal{O}}$ \cite{MR1620000} (see also
\cite[Corollary 2.18]{obloc}), a lift $\tilde{g}$ exists in the category
$\mathsf{LCA}_{\mathcal{O}}$,%
\[%
\xymatrix{
0 \ar[r] & G_1 \ar[r] \ar[d]_{g} & G_2 \ar[r] \ar@{-->}[dl]^{\tilde{g}}
& {G_2}/{G_1} \ar[r] & 0\\
& \mathbb{R}_{\nu},
}%
\]
Since both $G_{2}$ and $\mathbb{R}_{\nu}$ are $F$-vector spaces, the fullness
of $\mathsf{LCA}_{F}\rightarrow\mathsf{LCA}_{\mathcal{O}}$ (Lemma
\ref{lemma_LCAFtoLCAO}) implies that $\tilde{g}$ is also a lift in
$\mathsf{LCA}_{F}$. The property of being projective is proven dually. (2) Let
$W$ be a discrete $F$-vector space. Given a surjection $f:G\twoheadrightarrow
H$ and a morphism $g:W\rightarrow H$, by the surjectivity of $f$ we can find a
preimage of a basis of the set-theoretic image of $g$ in $G$. Define a lift
$\tilde{g}:W\rightarrow G$ by mapping $W/\ker g$ to these lifts. As we only
deal with discrete vector spaces, this lifts to a map $\tilde{g}$ defined on
all of $W$. As $W$ is discrete, any $F$-linear map originating from it is
necessarily continuous, so $\tilde{g}$ is a lift in $\mathsf{LCA}_{F}$. (3)
follows from Pontryagin duality.
\end{proof}

Next, we record some basic observations regarding the structure of objects in
$\mathsf{LCA}_{F}$.

\begin{definition}
\label{def_QuasiAdelicBlock}A \emph{quasi-adelic block} is an object
$Q\in\mathsf{LCA}_{F}$ which possesses a compact clopen $\mathcal{O}%
$-submodule $C $ such that it can be written as%
\[
Q=\bigoplus_{\sigma\in I}\mathbb{R}_{\sigma}\oplus H\qquad\text{with}\qquad
H=\bigcup_{n\geq1}\frac{1}{n}C\text{,}%
\]
for some finite list $I$ of real and complex places.

\begin{enumerate}
\item We call it an \emph{adelic block} if $C$ can be chosen such that
$\bigcap_{n\geq1}nC=0$.

\item We call it \emph{vector-free} if $I$ can be chosen empty, $I=\varnothing
$.
\end{enumerate}

Let $\mathsf{LCA}_{F,qab}$ be the full subcategory of $\mathsf{LCA}_{F}$ whose
objects are quasi-adelic blocks, $\mathsf{LCA}_{F,ab}$ the full subcategory of
adelic blocks, $\mathsf{LCA}_{F,vfab}$ the full subcategory of vector-free
adelic blocks.
\end{definition}

\begin{example}
The ad\`{e}les $\mathbb{A}$ are an adelic block.
\end{example}

\begin{remark}
We shall later prove a variant of the Braconnier--Vilenkin theorem in the
context of $\mathsf{LCA}_{F}$ and it will show that adelic blocks really more
or less look like variations of the ad\`{e}les of a number field, whence the
name. See Theorem \ref{thm_StrongBracVilenkin} below.
\end{remark}

\begin{definition}
A \emph{vector }$\mathcal{O}$\emph{-module} is an object in $\mathsf{LCA}_{F}
$ which is isomorphic to $\bigoplus_{\sigma\in I}\mathbb{R}_{\sigma}$ for some
finite list $I$ of real and complex places.
\end{definition}

It is equivalent to speak of objects $X\in\mathsf{LCA}_{F}$ whose underlying
additive group $(X;+)$ is isomorphic to $\mathbb{R}^{n}$ in\ $\mathsf{LCA}$
for some $n$. See \cite[Proposition 2.8]{obloc} for the proof. Recall that the
finiteness condition does not play the crucial r\^{o}le here: A product of
copies of $\mathbb{R}$ is locally compact if and only if it is finite.

We begin with a crude decomposition result for general objects in
$\mathsf{LCA}_{F}$:

\begin{theorem}
\label{thm_StructLCAF}Let $G\in\mathsf{LCA}_{F}$ be arbitrary. Then there
exists a clopen quasi-adelic block $Q\hookrightarrow G$ such that $G\cong
Q\oplus D$, where $D$ is a discrete $F$-vector space.
\end{theorem}

\begin{proof}
Firstly, regard $G$ as an object in $\mathsf{LCA}_{\mathcal{O}}$. By Levin's
structure result, \cite[Theorem 2]{MR0310125}, there exists an exact sequence%
\[
\bigoplus_{\sigma\in I}\mathbb{R}_{\sigma}\oplus C\hookrightarrow
G\twoheadrightarrow D\qquad\text{in}\qquad\mathsf{LCA}_{\mathcal{O}}%
\]
with $C$ a compact clopen $\mathcal{O}$-submodule in $G$, $D$ discrete, and
$I$ a finite list of real and complex places. Define $H:=\bigcup_{n\geq1}%
\frac{1}{n}C$ inside $G$. Since $G$ is a topological $F$-vector space, all
elements of $F^{\times}$ act as homeomorphisms on $G$, so all $\frac{1}{n}C$
are also compact and clopen $\mathcal{O}$-submodules. As $H$ is a union of
open $\mathcal{O}$-submodules, $H$ is itself an open submodule, thus clopen,
and since every element $a\in F$ can be written as $b/n$ with $b\in
\mathcal{O}$ and $n\geq1$ an integer, $H$ is also an $F$-vector subspace of
$G$. We claim that it is a topological $F$-submodule: To this end, note that
for all $a\in\mathcal{O}\setminus\{0\}$, the multiplication map $m_{a}%
:H\rightarrow H$ is continuous, it is surjective since $H$ is an $F$-vector
space, and $H$ is clearly $\sigma$-compact. Thus, by Pontryagin's Open Mapping
Theorem \cite[Theorem 3]{MR0442141}, $m_{a}$ is an open map. This shows that
multiplication with $\frac{1}{n}$, which is the inverse map $m_{n}^{-1}$, is
continuous.\footnote{See Example \ref{example_Robertson} to see that this is
by no means true anyway.} Thus, $H$ is a clopen topological $F$-submodule of
$G$. Define%
\[
Q:=\bigoplus_{\sigma\in I}\mathbb{R}_{\sigma}+H\subseteq G
\]
in $G$. This is an open $F$-submodule of $G$ (since the summand $H$ is open),
thus clopen. It is actually a direct sum: If $x$ lies in the intersection
$\bigoplus_{\sigma\in I}\mathbb{R}_{\sigma}\cap H$, then for $n$ sufficiently
large $nx$ lies in $\bigoplus_{\sigma\in I}\mathbb{R}_{\sigma}\cap C$, but
then we must have $nx=0$, proving $x=0$ as we are in an $F$-vector space, i.e.
there are no non-trivial torsion elements. Hence, we have constructed an exact
sequence $Q\hookrightarrow G\twoheadrightarrow D $ in $\mathsf{LCA}_{F}$, but
since $D$ is discrete (as $Q$ was open), it is a projective object by Lemma
\ref{Lem_InjProjInLCAF}, and thus the admissible epic admits a section in
$\mathsf{LCA}_{F}$. This proves the claim.
\end{proof}

In the previous proof we exploited that the scalar action of $\mathbb{Q}%
^{\times}$ is by homeomorphisms. Having this property is by no means automatic
for $\mathbb{Q}$-vector spaces which admit a locally compact topology:

\begin{example}
[\cite{MR0217211}, \cite{MR1173767}]\label{example_Robertson}We follow an idea
of L. Robertson. Fix a prime number $p$ and let $\omega$ be any cardinal.
Define, just as an untopologized abelian group,%
\begin{equation}
G:=\bigcup_{n\geq1}\left(  \frac{1}{p^{n}}\prod_{\omega}\mathbb{Z}_{p}\right)
\qquad\subseteq\qquad\prod_{\omega}\mathbb{Q}_{p}\text{.}\label{l_M1}%
\end{equation}
Equivalently: We consider the product of $\omega$ copies of $\mathbb{Q}_{p}$.
Then we let $G$ be the subgroup of those sequences $\underline{\lambda
}=(\lambda_{i})$ such that there exists some $N\geq1$ with%
\[
p^{N}\cdot\underline{\lambda}=(p^{N}\lambda_{i})\in\prod_{\omega}%
\mathbb{Z}_{p}\text{.}%
\]
It is clear that $G$ is an abelian group, and indeed a $\mathbb{Q}$-vector
space. Next, we topologize $G$ by declaring its subgroup $\prod_{\omega
}\mathbb{Z}_{p}$ a compact open. Then $G$ is locally compact, so
$G\in\mathsf{LCA}$. If $\omega$ is finite, then $G\cong\prod_{\omega
}\mathbb{Q}_{p}$, and moreover $G\in\mathsf{LCA}_{\mathbb{Q}}$ is an adelic
block. Consider the case of $\omega$ infinite: \textit{Assume} multiplication
by $\frac{1}{p}$ is continuous. Then multiplication by all $p^{n}$
($n\in\mathbb{Z}$) is a homeomorphism. Thus, as $\prod_{\omega}\mathbb{Z}_{p}$
is a compact open in the topology of $G$, it follows that all sets $\frac
{1}{p^{n}}\prod_{\omega}\mathbb{Z}_{p}$ are compact opens. Hence, by Equation
\ref{l_M1} the group $G$ is $\sigma$-compact. But all quotients of a $\sigma
$-compact group are also $\sigma$-compact, so $G/\prod_{\omega}\mathbb{Z}_{p}$
is $\sigma$-compact. However, as $\prod_{\omega}\mathbb{Z}_{p}$ is also an
open subgroup, this quotient is also discrete. A discrete $\sigma$-compact
group is countable. Now, inside this quotient we find%
\[
\frac{1}{p}\prod_{\omega}\mathbb{Z}_{p}\left/  \prod_{\omega}\mathbb{Z}%
_{p}\right.  \subset G\left/  \prod_{\omega}\mathbb{Z}_{p}\right.
\]
and the group on the left is isomorphic to $\prod_{\omega}\mathbb{F}_{p}$.
However, since $\omega$ is infinite, this is an uncountable set. It has
cardinality $2^{\omega}$ (indeed, for $p=2$ the set $\prod_{\omega}%
\mathbb{F}_{2}$ is literally the power set of a set of cardinality $\omega$).
Thus, we have found an uncountable set inside a supposedly countable group,
contradiction. It follows that $G\in\mathsf{LCA}_{\mathbb{Z}}$, that it is
algebraically a $\mathbb{Q}$-vector space, but it is not an object of
$\mathsf{LCA}_{\mathbb{Q}}$. In particular, $G$ is not an adelic block.
\end{example}

\begin{definition}
We write $\mathsf{LCA}_{F,qab+fd}$ for the full subcategory of $\mathsf{LCA}%
_{F}$ whose objects admit a presentation as $G\cong Q\oplus D$, where $Q$ is
quasi-adelic and $D$ a finite-dimensional discrete $F$-vector space.
\end{definition}

\begin{lemma}
\label{lemma_nomapsfromqabtodiscrete}Suppose $D\in\mathsf{LCA}_{F}$ is discrete.

\begin{enumerate}
\item Every morphism $h:G\rightarrow D$ from $G\in\mathsf{LCA}_{F,qab}$ must
be the zero map.

\item Every morphism $h:G\rightarrow D$ from $G\in\mathsf{LCA}_{F,qab+fd}$ has
finite-dimensional image.
\end{enumerate}
\end{lemma}

\begin{proof}
(1) Since $G$ is quasi-adelic, we can write $G=\bigoplus_{\sigma\in
I}\mathbb{R}_{\sigma}\oplus H$ with $H=\bigcup_{n\geq1}\frac{1}{n}C$, where
$I$ is a finite list of real and complex places and $C$ a compact clopen
$\mathcal{O}$-submodule of $G$. Since each $\mathbb{R}_{\sigma}$ is connected,
but $D$ discrete, the restriction of $h$ to the vector $\mathcal{O}$-module
summand must be zero. On the other hand, the image $h(C)$ is a compact
$\mathcal{O}$-submodule of $D$. As $D$ is discrete, $h(C)$ must be finite.
Hence, it is a torsion $\mathcal{O}$-module inside the $F$-vector space $D$,
and therefore $h(C)=0$. It follows that $h=0$. (2) The morphism has the shape%
\[
h:Q\oplus D^{\prime}\longrightarrow D\text{,}%
\]
where $Q$ is quasi-adelic, $D^{\prime}$ finite-dimensional discrete and by
part (1) of the proof, $\left.  h\mid_{Q}\right.  =0$, so the image of $h$
equals $h(D^{\prime})$ and since $D^{\prime}$ is finite-dimensional, so must
be $h(D^{\prime})$.
\end{proof}

\begin{example}
Compact $F$-vector spaces are injective objects in $\mathsf{LCA}_{F}$ by Lemma
\ref{Lem_InjProjInLCAF}, but not projective. To see this, consider the
ad\`{e}le sequence $F\hookrightarrow\mathbb{A}\twoheadrightarrow\mathbb{A}/F$,
where $\mathbb{A}/F\simeq F^{\vee}$ is a compact $F$-vector space. If it was
projective, the quotient map would split, so $\mathbb{A}\simeq F\oplus
F^{\vee}$. In particular, so would the initial map, giving a projection map
$\pi:\mathbb{A}\twoheadrightarrow F$, but by Lemma
\ref{lemma_nomapsfromqabtodiscrete} every such map must be the zero map,
giving a contradiction. Elaborating on this theme, the ad\`{e}le sequence
being non-split can be used as a characterization of the ad\`{e}les, see the
work of Levin \cite{MR0310125}.
\end{example}

\begin{proposition}
\label{prop_LCAFqab_leftfilt}Suppose $G\in\mathsf{LCA}_{F}$. For every
morphism $f:G^{\prime}\rightarrow G$ such that $G^{\prime}\in\mathsf{LCA}%
_{F,qab}$ (resp. $\mathsf{LCA}_{F,qab+fd}$), there exists a clopen submodule
$\tilde{Q}\in\mathsf{LCA}_{F,qab}$ (resp. $\mathsf{LCA}_{F,qab+fd}$) inside
$G$ and the dashed arrow making the diagram%
\[%
\xymatrix{
& \tilde{Q} \ar@{^{(}->}[d] \\
G^{\prime} \ar[r]_{f} \ar@{-->}[ur] & G
}%
\]
commute. In particular, the category $\mathsf{LCA}_{F,qab}$ (resp.
$\mathsf{LCA}_{F,qab+fd}$) is left filtering in $\mathsf{LCA}_{F}$.
\end{proposition}

\begin{proof}
We prove this for $G^{\prime}\in\mathsf{LCA}_{F,qab+fd}$ [and discuss the
variation for $G^{\prime}\in\mathsf{LCA}_{F,qab}$ in square brackets].
Consider $f:G^{\prime}\rightarrow G$. We apply Theorem \ref{thm_StructLCAF} to
$G$ and obtain a commutative diagram%
\begin{equation}%
\xymatrix{
& \hat{Q} \ar@{^{(}->}[d] \\
G^{\prime} \ar[r]_{f} \ar[dr]_{h} & G \ar@{->>}[d] \\
& D \text{,}
}%
\label{lcw1}%
\end{equation}
where the downward exact sequence is split, $\hat{Q}$ is quasi-adelic and $D$
discrete. By Lemma \ref{lemma_nomapsfromqabtodiscrete} the morphism $h$ has
finite-dimensional image. Post-composing the vertical quotient map with
$D\twoheadrightarrow D/\operatorname*{im}(h)$, and defining $\tilde{Q}$ as the
kernel of the composite map, we obtain a new commutative diagram%
\begin{equation}%
\xymatrix{
& \tilde{Q} \ar@{^{(}->}[d] \\
G^{\prime} \ar@{-->}[ur] \ar[r]_{f} \ar[dr]_{0} & G \ar@{->>}[d] \\
& D/{\operatorname{im}(h)} \text{.}
}%
\label{lcw2}%
\end{equation}
By the universal property of kernels, there is a morphism $c:\hat
{Q}\rightarrow\tilde{Q}$ (intuitive slogan: the new kernel is at least as big
as the kernel beforehand):%
\[%
\xymatrix{
\hat{Q} \ar[r]^{c} \ar@{^{(}->}[dr] & \tilde{Q} \ar@{^{(}->}[d] \\
& G\text{.}
}%
\]
Since $\mathsf{LCA}_{F}$ is quasi-abelian, it is idempotent complete and thus
by \cite[dual to Proposition 7.6]{MR2606234} this morphism $c$ must be an
admissible monic. Hence, we get admissible monics $\hat{Q}\overset
{c}{\hookrightarrow}\tilde{Q}\hookrightarrow G$ and by\ Noether's Lemma,
\cite[Lemma 3.5]{MR2606234}, there is an exact sequence%
\[
\tilde{Q}/\hat{Q}\hookrightarrow G/\hat{Q}\twoheadrightarrow G/\tilde
{Q}\text{.}%
\]
Inspecting Diagrams \ref{lcw1} and \ref{lcw2}, the latter two terms are just
$D\twoheadrightarrow D/\operatorname*{im}(h)$ and therefore $\tilde{Q}/\hat
{Q}\cong\operatorname*{im}(h)$. Hence, we get the short exact sequence%
\[
\hat{Q}\hookrightarrow\tilde{Q}\twoheadrightarrow\operatorname*{im}(h)\text{.}%
\]
Since $\operatorname*{im}(h)$ is discrete, it is a projective object by Lemma
\ref{Lem_InjProjInLCAF} and therefore the admissible epic splits, giving us an
isomorphism $\tilde{Q}\cong\hat{Q}\oplus\operatorname*{im}(h)$. Since $\hat
{Q}$ is quasi-adelic and $\operatorname*{im}(h)$ a finite-dimensional discrete
$F$-vector space, we conclude that $\tilde{Q}\in\mathsf{LCA}_{F,qab+fd}$. [If
$G^{\prime}\in\mathsf{LCA}_{F,qab}$, then Lemma
\ref{lemma_nomapsfromqabtodiscrete} shows that $\operatorname*{im}(h)=0$, and
thus we obtain the stronger conclusion that $\tilde{Q}\cong\hat{Q}$, and thus
$\tilde{Q}\in\mathsf{LCA}_{F,qab}$.]\newline
\end{proof}

\begin{proposition}
\label{prop_lcafqab_extensionclosed}The category $\mathsf{LCA}_{F,qab}$ (resp.
$\mathsf{LCA}_{F,qab+fd}$) is extension-closed in $\mathsf{LCA}_{F}$.
\end{proposition}

\begin{proof}
Suppose%
\[
G^{\prime}\hookrightarrow G\twoheadrightarrow G^{\prime\prime}%
\]
is an exact sequence in $\mathsf{LCA}_{F}$. We prove the claim for the
situation $G^{\prime},G^{\prime\prime}\in\mathsf{LCA}_{F,qab+fd}$ [and discuss
the variation for $G^{\prime},G^{\prime\prime}\in\mathsf{LCA}_{F,qab}$ in
square brackets]. We apply Proposition \ref{prop_LCAFqab_leftfilt} to
$G^{\prime}\hookrightarrow G$ and obtain a commutative diagram%
\begin{equation}%
\xymatrix{
& \tilde{Q} \ar@{^{(}->}[d] \\
G^{\prime} \ar@{^{(}->}[r] \ar@{-->}[ur]^-{y} \ar[dr]_{0} & G \ar@
{->>}[d] \ar@{->>}[r] & G^{\prime\prime} \\
& D\text{,}
}%
\label{lcca1}%
\end{equation}
where $\tilde{Q}\in\mathsf{LCA}_{F,qab+fd}$ [resp. $\tilde{Q}\in
\mathsf{LCA}_{F,qab}$] and $D$ is discrete since it is the quotient by
something open. By \cite[dual to Proposition 7.6]{MR2606234} the lift $y$ must
be an admissible monic in $\mathsf{LCA}_{F}$. Hence, we get admissible monics
$G^{\prime}\hookrightarrow\tilde{Q}\hookrightarrow G$ and by\ Noether's Lemma
(\cite[Lemma 3.5]{MR2606234}), there is an exact sequence $\tilde{Q}%
/G^{\prime}\hookrightarrow G/G^{\prime}\twoheadrightarrow G/\tilde{Q}$. By
Diagram \ref{lcca1} we have $G/G^{\prime}\cong G^{\prime\prime}$ and
$G/\tilde{Q}\cong D$, so this sequence is isomorphic to%
\begin{equation}
\tilde{Q}/G^{\prime}\hookrightarrow G^{\prime\prime}\twoheadrightarrow
D\text{.}\label{lcw3}%
\end{equation}
However, since $G^{\prime\prime}\in\mathsf{LCA}_{F,qab+fd}$ and $D$ is
discrete, Lemma \ref{lemma_nomapsfromqabtodiscrete} shows that the admissible
epic has finite-dimensional image. Being an epic, we learn that $G/\tilde
{Q}\cong D$ is a finite-dimensional discrete vector space. In particular,
again by Lemma \ref{Lem_InjProjInLCAF}, it is a projective object, so the
exact sequence $\tilde{Q}\hookrightarrow G\twoheadrightarrow G/\tilde{Q}$
splits, i.e. we obtain an isomorphism%
\[
G\cong\tilde{Q}\oplus G/\tilde{Q}\text{,}%
\]
where $\tilde{Q}\in\mathsf{LCA}_{F,qab+fd}$ and $G/\tilde{Q}$ is a
finite-dimensional discrete $F$-vector space. Thus, $G\in\mathsf{LCA}%
_{F,qab+fd}$, proving our claim. [If we have $G^{\prime},G^{\prime\prime}%
\in\mathsf{LCA}_{F,qab}$, then we have $\tilde{Q}\in\mathsf{LCA}_{F,qab}$.
Moreover, in Equation \ref{lcw3} the middle term $G^{\prime\prime}$ is
quasi-adelic and $D$ discrete, so by Lemma \ref{lemma_nomapsfromqabtodiscrete}
the admissible epic must be the zero map. Being an epic, this forces that
$D=0$ and hence the vertical sequence in Diagram \ref{lcca1} implies that
$\tilde{Q}\hookrightarrow G$ is an isomorphism. Since we already know that
$\tilde{Q}\in\mathsf{LCA}_{F,qab}$, this proves $G\in\mathsf{LCA}_{F,qab}$,
finishing the proof.]
\end{proof}

\begin{corollary}
The full subcategory $\mathsf{LCA}_{F,qab}$ (resp. $\mathsf{LCA}_{F,qab+fd}$)
is a fully exact subcategory of $\mathsf{LCA}_{F}$, in particular it carries a
canonical induced exact category structure.
\end{corollary}

\begin{proof}
This follows from being closed under extensions, cf. \cite[Lemma
10.20]{MR2606234}.
\end{proof}

\begin{remark}
While being fully exact categories, neither of them is closed under Pontryagin duality.
\end{remark}

\begin{example}
\label{Example_LCAFqabDoesNotHaveKernels}Unlike $\mathsf{LCA}_{F}$, the exact
category $\mathsf{LCA}_{F,qab}$ is not quasi-abelian. To this end, consider
the ad\`{e}le sequence $F\hookrightarrow\mathbb{A}\twoheadrightarrow
\mathbb{A}/F$. Then $\mathbb{A}$ and $\mathbb{A}/F$ are quasi-adelic. If
$\mathsf{LCA}_{F,qab}$ was closed under kernels, $F$ would be quasi-adelic
itself. We apply Lemma \ref{lemma_nomapsfromqabtodiscrete} to the identity map
of $F$ and conclude $F=0$, which is a contradiction.
\end{example}

Even though kernels do not exist in $\mathsf{LCA}_{F,qab}$ in general,
idempotents have kernels:

\begin{proposition}
\label{prop_LCAqabIsIdempotentComplete}The exact category $\mathsf{LCA}%
_{F,qab}$ (resp. $\mathsf{LCA}_{F,qab+fd}$) is idempotent complete, i.e. every
idempotent has a kernel.
\end{proposition}

We prove this also for $\mathsf{LCA}_{F,qab+fd}$, since it causes no extra
effort to handle it as well. Nevertheless, a little below we shall show that
the latter category indeed has \textit{all} kernels.

\begin{proof}
Suppose $G\in\mathsf{LCA}_{F,qab+fd}$ and $p:G\rightarrow G$ is an idempotent,
$p^{2}=p$. We work in the bigger category $\mathsf{LCA}_{F}$. It is
quasi-abelian and thus idempotent complete. Hence, in this bigger category the
idempotents $p$ and $1-p$ have a kernel, call them $A$ and $B$, and we get a
corresponding direct sum decomposition%
\begin{equation}
G\cong A\oplus B\qquad\text{in}\qquad\mathsf{LCA}_{F}\text{.}\label{laa6}%
\end{equation}
By Theorem \ref{thm_StructLCAF} there is an exact sequence%
\begin{equation}
Q\hookrightarrow A\overset{q}{\twoheadrightarrow}D\label{laa5}%
\end{equation}
in $\mathsf{LCA}_{F}$ with $Q$ quasi-adelic and $D$ discrete. Thus, we get the
diagram%
\[%
\xymatrix{
& Q \oplus B \ar@{^{(}->}[d] \\
G \ar[r]^-{\simeq} \ar[dr]_{h} & A \oplus B \ar@{->>}[d]^{q} \\
& D
}%
\]
and by Lemma \ref{lemma_nomapsfromqabtodiscrete} the image of $h$ is
finite-dimensional. As the quotient map $q$ is an epic, it follows that $D$ is
itself a finite-dimensional discrete $F$-vector space. Hence, as $D$ is also a
projective object by Lemma \ref{Lem_InjProjInLCAF}, Equation \ref{laa5}
implies that $A$ is isomorphic to $Q\oplus D$ with $D$ finite-dimensional
discrete, so $A\in\mathsf{LCA}_{F,qab+fd}$. Now, repeat the same argument for
$B$. Thus, both sides of the isomorphism in Equation \ref{laa6} lie in
$\mathsf{LCA}_{F,qab+fd}$ and since this is a full subcategory, this
isomorphism in $\mathsf{LCA}_{F}$ is actually an isomorphism already in
$\mathsf{LCA}_{F,qab+fd}$. [If $G\in\mathsf{LCA}_{F,qab}$, Lemma
\ref{lemma_nomapsfromqabtodiscrete} implies that $h$ must be the zero map, so
in this case $D=0$. Hence, the conclusion of the proof even shows that $A=Q$,
so $A\in\mathsf{LCA}_{F,qab}$, and analogously for $B $.]
\end{proof}

A very similar argument shows that quasi-adelic blocks are closed under
cokernels in $\mathsf{LCA}_{F}$:

\begin{lemma}
\label{lemma_Fqab_closed_under_cokernels}The full subcategory $\mathsf{LCA}%
_{F,qab}$ is closed under cokernels of admissible monics in $\mathsf{LCA}_{F}
$.
\end{lemma}

\begin{proof}
Let $f:G^{\prime}\hookrightarrow G$ be an admissible monic with respect to the
exact structure of $\mathsf{LCA}_{F}$, but with $G^{\prime},G\in
\mathsf{LCA}_{F,qab}$. Then in $\mathsf{LCA}_{F}$ we have an exact sequence
$G^{\prime}\hookrightarrow G\twoheadrightarrow G^{\prime\prime}$ with
$G^{\prime\prime}\in\mathsf{LCA}_{F}$ and applying Theorem
\ref{thm_StructLCAF} we obtain a commutative diagram%
\[%
\xymatrix{
& & Q \ar@{^{(}->}[d] \\
G^{\prime} \ar@{^{(}->}[r] & G \ar@{->>}[r] \ar[dr]_{g} \ar@{-->}[ur]^{h}
& G^{\prime\prime} \ar@{->>}[d] \\
& & D\text{,}
}%
\]
where $Q$ is quasi-adelic and $D$ discrete. By Lemma
\ref{lemma_nomapsfromqabtodiscrete} the morphism $g$ must be the zero map, but
since it is also an epic, it follows that $D=0$. Hence, $G^{\prime\prime}\cong
Q$, and thus $G^{\prime\prime}$ is quasi-adelic.
\end{proof}

We have already seen in Example \ref{Example_LCAFqabDoesNotHaveKernels} that
the exact category $\mathsf{LCA}_{F,qab}$ does not have all kernels. We shall
prove that this issue disappears in the slightly bigger category
$\mathsf{LCA}_{F,qab+fd}$.

\begin{theorem}
[Kernel Theorem]\label{thm_KernelOfqabfds}Suppose $f:G\rightarrow G^{\prime}$
with $G\in\mathsf{LCA}_{F,qab+fd}$ and $G^{\prime}\in\mathsf{LCA}_{F}$. Then
the kernel of $f$ in $\mathsf{LCA}_{F}$ lies in the full subcategory
$\mathsf{LCA}_{F,qab+fd}$.
\end{theorem}

An equivalent formulation is that the full subcategory $\mathsf{LCA}%
_{F,qab+fd}$ is closed under subobjects.

\begin{proof}
By assumption $G$ has a presentation as $G\cong Q\oplus D$ with $Q$
quasi-adelic and $D$ a finite-dimensional discrete $F$-vector space. Write%
\[
G=\underset{Q}{\underbrace{\bigoplus_{\sigma\in I}\mathbb{R}_{\sigma}%
\oplus\bigcup_{n\geq1}\frac{1}{n}C}}\oplus\underset{D}{\underbrace
{\bigoplus_{w=1}^{\dim_{F}D}\bigcup_{n\geq1}\frac{1}{n}\mathcal{O}}}\text{,}%
\]
where $I$ is a finite list of real and complex places, $C$ a compact clopen
$\mathcal{O}$-submodule of $Q$, and in the last summand we exploit that
$\bigcup_{n\geq1}\frac{1}{n}\mathcal{O}=\mathcal{O}\otimes_{\mathbb{Z}%
}\mathbb{Q}=F$. For the first two summands we have just used the definition of
quasi-adelic blocks. For each $n\geq1$ we define%
\begin{equation}
K_{n}:=\ker\left(  \bigoplus_{\sigma\in I}\mathbb{R}_{\sigma}\oplus\frac{1}%
{n}C\oplus\bigoplus_{w=1}^{\dim_{F}D}\frac{1}{n}\mathcal{O}\overset
{f}{\longrightarrow}G^{\prime}\right)  \text{,}\label{lcwa1}%
\end{equation}
where the kernel is taken in $\mathsf{LCA}_{\mathcal{O}}$ (note that the
morphism is the restriction of $f$ to an $\mathcal{O}$-submodule, so we can
only regard it in $\mathsf{LCA}_{\mathcal{O}}$ anyway). Since the category
$\mathsf{LCA}_{\mathcal{O}}$ is quasi-abelian, Proposition
\ref{prop_LCARisQuasiAbelianExact}, this kernel exists. We claim that%
\[
K_{n}=\frac{1}{n}K_{1}\text{,}%
\]
where the multiplication with $\frac{1}{n}$ is understood as using the
$F$-vector space structure of $G$, and we refer to the equality of
$\mathcal{O}$-submodules inside $G$. The proof is trivial: Every $x\in K_{n}$
has the shape $x=r+\frac{1}{n}c$ with $r\in\bigoplus_{\sigma\in I}%
\mathbb{R}_{\sigma}$ and $c\in C\oplus\mathcal{O}^{\oplus\dim_{F}D}$. Rewrite
it as $x=\frac{1}{n}(n\cdot r+c)$, with $nr+c\in K_{1}$, so $K_{n}%
\subseteq\frac{1}{n}K_{1}$, and conversely rewrite $\frac{1}{n}(r+c)\in
\frac{1}{n}K_{1}$ as $(\frac{1}{n}r)+\frac{1}{n}c\in K_{n}$. Next, define%
\[
\hat{K}:=\ker\left(  G\overset{f}{\longrightarrow}G^{\prime}\right)
\]
in $\mathsf{LCA}_{\mathcal{O}}$ (we get the same object if we take the kernel
in $\mathsf{LCA}_{F}$ and apply the forgetful functor). Clearly $\hat{K}$ is
the union of all $K_{n}$, so by the previous observation we get%
\begin{equation}
\hat{K}=\bigcup_{n\geq1}\frac{1}{n}K_{1}\label{lcwa2}%
\end{equation}
in $\mathsf{LCA}_{\mathcal{O}}$. Next, inspecting Equation \ref{lcwa1} more
closely for $n=1$, we observe that the source of $f$,%
\[
X:=\bigoplus_{\sigma\in I}\mathbb{R}_{\sigma}\oplus C\oplus\bigoplus
_{w=1}^{\dim_{F}D}\mathcal{O}\text{,}%
\]
is compactly generated, i.e. it lies in $\mathsf{LCA}_{\mathcal{O},cg}$
\cite{obloc} (take a compact unit ball in the vector group part $U$, and a
basis $b_{1},\ldots,b_{\dim_{F}D}$ of $\mathcal{O}$; then the union of the
compact set $U\times C\times\{b_{1},\ldots,b_{\dim_{F}D}\}$ with its negative
is a compact symmetric neighbourhood of the identity which compactly generates
all of $X$). The quality to be compactly generated only depends on the
underlying LCA group. The kernel $K_{1}$ is a subobject of $X$, and thus
itself compactly generated by \cite[Theorem 2.6 (2)]{MR0215016}. Thus, by
Levin's structure result, \cite[Proposition 1]{MR0310125}, there exists an
isomorphism%
\[
K_{1}\cong\bigoplus_{\nu\in I^{\prime}}\mathbb{R}_{\nu}\oplus\bigoplus
_{J\in\mathcal{I}}J\oplus\tilde{C}\qquad\text{in}\qquad\mathsf{LCA}%
_{\mathcal{O}}\text{,}%
\]
where $I^{\prime}$ is a finite list of real and complex places, $\mathcal{I} $
is a finite(!) list of ideals of $\mathcal{O}$, and $\tilde{C}$ a compact
$\mathcal{O}$-module. From Equation \ref{lcwa2} we obtain%
\[
\hat{K}=\bigcup_{n\geq1}\frac{1}{n}K_{1}=\bigoplus_{\nu\in I^{\prime}%
}\mathbb{R}_{\nu}\oplus\bigoplus_{J\in\mathcal{I}}\left(  \bigcup_{n\geq
1}\frac{1}{n}J\right)  \oplus\bigcup_{n\geq1}\frac{1}{n}\tilde{C}\text{.}%
\]
We have $\bigcup_{n\geq1}\frac{1}{n}J=J\otimes_{\mathbb{Z}}\mathbb{Q}\cong F$,
and since $\mathcal{I}$ is a \textit{finite} list of ideals, this simplifies
to%
\[
\cong\bigoplus_{\nu\in I^{\prime}}\mathbb{R}_{\nu}\oplus F^{\#\mathcal{I}%
}\oplus H\text{,}%
\]
where $H:=\bigcup_{n\geq1}\frac{1}{n}\tilde{C}$ is clearly quasi-adelic. We
observe that $\hat{K}\in\mathsf{LCA}_{F,qab+fd}$. This finishes the proof.
\end{proof}

\begin{definition}
We write $\mathsf{LCA}_{F,C}$ for the full subcategory of $\mathsf{LCA}_{F}$
of compact $F$-vector spaces.
\end{definition}

\begin{lemma}
\label{lemma_LCAFCIsConnected}Every object in $\mathsf{LCA}_{F,C}$ is connected.
\end{lemma}

\begin{proof}
Note that such an object $V$ is compact, so $V$ is connected if and only if
$V^{\vee}$ is torsion-free \cite[Corollary 4 to Theorem 30]{MR0442141}.
However, $V^{\vee}$ is a discrete $F$-vector space and thus clearly
torsion-free. Hence, $V$ is connected.
\end{proof}

\begin{example}
This differs drastically from the category of compact $\mathcal{O}$-modules
$\mathsf{LCA}_{\mathcal{O},C}$. For example, for $F=\mathbb{Q}$ the latter
contains $\mathbb{Z}_{p}$, which is neither an injective object, nor
connected. In fact it is totally disconnected.
\end{example}

Next, the quasi-adelic blocks have a further decomposition:

\begin{proposition}
\label{prop_DecomposeQuasiAdelicBlock}Let $Q$ be a quasi-adelic block.

\begin{enumerate}
\item Then $Q$ has a canonical decomposition%
\[
Q\cong V\oplus A\qquad\text{in}\qquad\mathsf{LCA}_{F}\text{,}%
\]
where $V$ is a compact $F$-vector space and $A$ an adelic block. Concretely,
$V=\bigcap_{n\geq1}nC$, where $C$ is chosen as in Definition
\ref{def_QuasiAdelicBlock}.

\item Moreover, there is also a canonical decompositon $A\cong\bigoplus
_{\sigma\in I}\mathbb{R}_{\sigma}\oplus A^{\prime}$, where $I$ is a finite
list of real and complex places, and $A^{\prime}$ is a vector-free adelic block.
\end{enumerate}
\end{proposition}

\begin{proof}
As $Q$ is a quasi-adelic block, following the notation of Definition
\ref{def_QuasiAdelicBlock} we may write it as $Q=\bigoplus_{\sigma\in
I}\mathbb{R}_{\sigma}\oplus H$ for $I$ a finite list of real and complex
places. Define $N:=\#I$.\newline\textit{(Step 1)\ }We first prove the claim
for $N=0$: Define $V:=\bigcap_{n\geq1}nC$. As $C$ is compact, each $nC$ is
compact, and thus $V $ is a closed subgroup. We have $V\subseteq C$, so it is
a closed subset of a compactum, thus itself compact. Each $v\in V$ can be
written as $v=n\cdot c_{n}$ for all $n\geq1$ and $c_{n}\in C$ by assumption,
and so for all $\alpha\in\mathcal{O}$, $\alpha v=n\cdot(\alpha c_{n})$ with
$\alpha c_{n}\in C$, i.e. $V$ is a closed $\mathcal{O}$-submodule. Similarly,
we get $\frac{1}{m}v=nc_{nm}$ for all $m\geq1$, so it is a closed
$F$-submodule of $Q$. Thus, we get an exact sequence $V\hookrightarrow
Q\overset{q}{\twoheadrightarrow}A$ in $\mathsf{LCA}_{F}$ and it remains to
show that $A$ is an adelic block. As the quotient map $q$ is surjective and
$Q=\bigcup_{n\geq1}\frac{1}{n}C$, we get
\[
A=\bigcup_{n\geq1}\frac{1}{n}q(C)\text{.}%
\]
Moreover, $q(C)$ is also compact and since quotient maps are open, $q(C)$ is
also clopen. Finally, let $x\in\bigcap_{n\geq1}nq(C)$. For a lift $\tilde{x}$
of $x$ to $Q$ this means that for all $n\geq1$ there exists $c_{n}\in C$ and
$v_{n}\in V$ such that $\tilde{x}=nc_{n}+v_{n}$. For all $v_{n}\in V$ there
exists $w_{n}\in C$ such that $v_{n}=nw_{n}$, hence $\tilde{x}=n(c_{n}+w_{n})$
with $c_{n},w_{n}\in C$. Thus, $\tilde{x}\in V$, i.e. $x=0$ in $A$. It follows
that $A$ is an adelic block.\newline\textit{(Step 2) }It remains to deal with
the cases $N\geq1$. We do this by induction over $N$. Suppose all cases with
$\#I<N$ are settled. Since $N\geq1$, $Q$ has a summand of the shape
$\mathbb{R}_{\nu}$, say $Q=\tilde{Q}\oplus\mathbb{R}_{\nu}$, and by induction
hypothesis, we have a decomposition $\tilde{V}\hookrightarrow\tilde
{Q}\twoheadrightarrow\tilde{A} $ as in our claim and thus%
\[
\tilde{V}\hookrightarrow\tilde{Q}\oplus\mathbb{R}_{\nu}\twoheadrightarrow
\tilde{A}\oplus\mathbb{R}_{\nu}%
\]
meets the requirements of our claim, and moreover $\tilde{V}=\bigcap_{n\geq
1}nC=V$ did not change.\newline\textit{(Step 3)} Thus, we obtain a sequence
$V\hookrightarrow Q\overset{q}{\twoheadrightarrow}A$ with $A$ adelic and $V$ a
compact $F$-vector space. By Lemma \ref{Lem_InjProjInLCAF} the object $V$ is
injective, so the admissible monic admits a section. Hence, we get an
isomorphism $Q\cong V\oplus A$ as in our claim.\newline\textit{(Step 4)}\ In
order to obtain the alternative presentation, note that all vector
$\mathcal{O}$-module summands in the adelic block $A$ are projective objects
(Lemma \ref{Lem_InjProjInLCAF}), so the surjection onto them splits, and hence
they are direct summands of $A$.
\end{proof}

Next, we recall a classical concept of locally compact abelian groups. Its
variant for $\mathcal{O}$-modules was introduced by Levin.

\begin{definition}
[\cite{MR0310125}]\label{def_TopTors}$G$ is a\emph{\ topological torsion}
$\mathcal{O}$-module if it is the union of all its compact $\mathcal{O}%
$-submodules, and the intersection of all its compact $\mathcal{O}$-submodules
is trivial.
\end{definition}

We can relate this property to adelic blocks.

\begin{lemma}
\label{lemma_StructVectorFreeAdelicBlock}Suppose $G\in\mathsf{LCA}_{F,ab}$ is
a vector-free adelic block. Then the following holds:

\begin{enumerate}
\item $G$ and $G^{\vee}$ are topological torsion $\mathcal{O}$-modules.

\item $G$ and $G^{\vee}$ are totally disconnected.
\end{enumerate}
\end{lemma}

\begin{proof}
By definition we have a compact clopen $\mathcal{O}$-submodule $C$ in $G$. As
$G$ is a continuous $F$-vector space, all elements of $F^{\times}$ act as
homeomorphisms on $G$. Thus, all $\mathcal{O}$-submodules $\frac{1}{n}C$ are
also clopen and compact. It follows that the union of all compact
$\mathcal{O}$-submodules is all of $G$. Moreover, $\bigcap_{n\geq1}nC=0$
implies that there is a collection of open submodules which intersect
trivially, thus the totality of all open submodules intersects trivially.
Hence, $G$ is topologically torsion in the sense of Definition
\ref{def_TopTors}. The underlying LCA group of a topological torsion module is
a topological torsion group in the sense that $\lim_{n\rightarrow\infty}n!g=0
$ for all $g\in G$. By a theorem of Robertson this implies that both $G$ and
$G^{\vee}$ are totally disconnected, \cite[(3.15)\ Theorem]{MR0217211} and
\cite[(3.16)\ Corollary]{MR0217211} implies that $G^{\vee}$ is also
topological torsion. This proves (1) and (2).
\end{proof}

Later, we shall prove the converse implication as well, see Proposition
\ref{prop_TopTorsionIsVectorFreeAdelicBlock}.

\begin{example}
We had shown in Proposition \ref{prop_LCAFqab_leftfilt} that the category of
quasi-adelic blocks is left filtering in $\mathsf{LCA}_{F}$. Adelic blocks are
\emph{not} left filtering. To this end, consider the morphism $f:\mathbb{A}%
\twoheadrightarrow\mathbb{A}/F$. The ad\`{e}les are an adelic block, so if
adelic blocks $\mathsf{LCA}_{F,ab}$ are left filtering, then there exists a
factorization%
\[
\mathbb{A}\overset{a}{\longrightarrow}A\hookrightarrow\mathbb{A}/F\text{,}%
\]
where $A$ is an adelic block. However, $\mathbb{A}/F$ is a compact $F$-vector
space, so all closed $F$-vector subspaces are themselves compact $F$-vector
spaces. Hence, the adelic block $A$ is vector-free and thus by Lemma
\ref{lemma_StructVectorFreeAdelicBlock} we learn that $A$ is totally
disconnected. It follows that the arrow $a$ maps the connected component
$\bigoplus_{\sigma\in I}\mathbb{R}_{\sigma}$ of the ad\`{e}les to zero, and
this contradicts the fact that the kernel of the composition $\ker
(\mathbb{A}\twoheadrightarrow\mathbb{A}/F)=F$ is much smaller.
\end{example}

\begin{proposition}
\label{prop_LCAFCLeftFilteringInLCAFQab}The full subcategory $\mathsf{LCA}%
_{F,C}$ is left $s$-filtering in $\mathsf{LCA}_{F,qab}$.
\end{proposition}

\begin{proof}
\textit{(Step 1)}\ It is clear that $\mathsf{LCA}_{F,C}$ is left filtering:
Given any $f:C\rightarrow G$ with $C\in\mathsf{LCA}_{F,C}$ and $G\in
\mathsf{LCA}_{F}$, the set-theoretic image of the compactum $C$ is again
compact, and thus necessarily closed in $G$ since the latter in Hausdorff.
This yields the required factorization%
\[%
\xymatrix{
& \operatorname{im}(f) \ar@{^{(}->}[d] \\
C \ar[r]_{f} \ar[ur] & G\text{.}
}%
\]
\textit{(Step 2) }Next, we show that the inclusion of categories is left
special. Suppose $G\twoheadrightarrow C$ is an admissible epic in
$\mathsf{LCA}_{F,qab}$ with $C\in\mathsf{LCA}_{F,C}$. Proposition
\ref{prop_DecomposeQuasiAdelicBlock} produces a commutative diagram, as
depicted below on the left%
\[%
\xymatrix{
V \ar@{^{(}->}[d] \ar[dr]^{h} \\
G \ar@{->>}[r] \ar@{->>}[d] & C \\
A
}%
\qquad\qquad\qquad\qquad%
\xymatrix{
V \ar@{^{(}->}[d] \ar[dr]^{0} \\
G \ar@{->>}[r] \ar@{->>}[d] & C/{\operatorname{im}(h)} \\
A \ar@{-->>}[ur]_{w}
}%
\]
with $V$ a compact $F$-vector space and $A$ an adelic block. As $V$ is
compact, its image under $h$ is closed and we can quotient it out, giving the
new diagram depicted on the right. The dashed arrow $w$ exists by the
universal property of cokernels. Moreover, it is an admissible epic in
$\mathsf{LCA}_{F,qab}$ by \cite[Proposition 7.6]{MR2606234}. To apply the
latter, we need to know that $\mathsf{LCA}_{F,qab}$ is idempotent complete,
which we had shown in Proposition \ref{prop_LCAqabIsIdempotentComplete}. It
follows that in $\mathsf{LCA}_{F,qab}$ there exists an exact sequence%
\[
K\hookrightarrow A\overset{w}{\twoheadrightarrow}C/\operatorname*{im}%
(h)\text{,}%
\]
where $K\in\mathsf{LCA}_{F,qab}$ is the kernel of $w$. As $A$ is an adelic
block, we can present it as $\bigoplus_{\sigma}\mathbb{R}_{\sigma}%
\oplus\bigcup_{n\geq1}\frac{1}{n}\tilde{C}$, where $\sigma$ runs through a
finite list of real and complex places, while $\tilde{C}$ is a compact clopen
$\mathcal{O}$-submodule of $A$ such that $\bigcap_{n\geq1}n\tilde{C}=0$. In
the quasi-abelian category $\mathsf{LCA}_{\mathcal{O}}$ we may define
\[
K_{n}:=\ker\left(  \bigoplus_{\sigma}\mathbb{R}_{\sigma}\oplus\frac{1}%
{n}\tilde{C}\overset{w}{\twoheadrightarrow}C/\operatorname*{im}(h)\right)
\text{.}%
\]
Then $K=\bigcup_{n\geq1}K_{n}$, and moreover $K_{n}=\frac{1}{n}K_{1}$ (this is
precisely as in the proof of Theorem \ref{thm_KernelOfqabfds}). As
$\bigoplus_{\sigma}\mathbb{R}_{\sigma}\oplus\frac{1}{n}\tilde{C}$ is compactly
generated and $K_{1}$ is a closed subgroup, $K_{1}$ is also a compactly
generated $\mathcal{O}$-module, \cite[Theorem 2.6 (2)]{MR0215016}. As in the
proof of Theorem \ref{thm_KernelOfqabfds} we obtain an isomorphism%
\begin{equation}
K\cong\bigoplus_{\nu\in I^{\prime}}\mathbb{R}_{\nu}\oplus F^{N}\oplus
H\text{,}\label{lpp5}%
\end{equation}
where $H:=\bigcup_{n\geq1}\frac{1}{n}\widehat{C}$ is quasi-adelic (for
$\widehat{C}$ some compact clopen $\mathcal{O}$-submodule, see the cited
proof), $N\geq0$ some integer and $I^{\prime}$ a finite list of real and
complex places. Firstly, we claim that $N=0$. To see this, consider the
projector $\pi:K\twoheadrightarrow F^{N}$ on the respective summand. As
$K\in\mathsf{LCA}_{F,qab}$, the morphism $\pi$ maps from a quasi-adelic block
to something discrete. Hence, $\pi=0$ by Lemma
\ref{lemma_nomapsfromqabtodiscrete}. Since $\pi$ was surjective, this forces
$N=0$. Furthermore, we claim that $\bigcap_{n\geq1}n\widehat{C}=0$. Suppose
not. Then $K$ has a compact $F$-vector space summand (by Proposition
\ref{prop_DecomposeQuasiAdelicBlock}), which is connected by Lemma
\ref{lemma_LCAFCIsConnected}. Thus, under the inclusion%
\[
K\hookrightarrow A=\bigoplus_{\sigma}\mathbb{R}_{\sigma}\oplus\bigcup_{n\geq
1}\frac{1}{n}\tilde{C}%
\]
this summand must map to $\bigoplus_{\sigma}\mathbb{R}_{\sigma}$ (which is the
conneced component of the image since the latter is an adelic block, so the
vector-free summand is totally disconnected by Lemma
\ref{lemma_StructVectorFreeAdelicBlock}). However, as this summand is also
compact, its image in $\bigoplus_{\sigma}\mathbb{R}_{\sigma}$ must be trivial.
Being an injection, this means that no such summand can exist, proving our
claim. It follows that $K$ is an adelic block. Using the same connectedness
argument, $\bigoplus_{\nu\in I^{\prime}}\mathbb{R}_{\nu}$ in $K$ (as in
Equation \ref{lpp5}), being connected, can only map to $\bigoplus_{\sigma
}\mathbb{R}_{\sigma}$. However, $K\hookrightarrow A\overset{w}%
{\twoheadrightarrow}C/\operatorname*{im}(h)$ is an exact sequence and the
quotient is compact, which forces by a simple dimension argument that the map
of connected components $\bigoplus_{\nu}\mathbb{R}_{\nu}\rightarrow
\bigoplus_{\sigma}\mathbb{R}_{\sigma}$ must be an isomorphism, for otherwise
the quotient would be non-compact. It follows that%
\[
\bigcup_{n\geq1}\frac{1}{n}\tilde{C}%
\]
alone surjects onto $C/\operatorname*{im}(h)$. However, being a vector-free
adelic block, this is totally disconnected (again Lemma
\ref{lemma_StructVectorFreeAdelicBlock}), and all quotient groups of totally
disconnected groups are totally disconnected.\newline The category of compact
$F$-vector spaces is Pontryagin dual to discrete $F$-vector spaces, and thus
an abelian category: As $C/\operatorname*{im}(h)$ is connected (Lemma
\ref{lemma_LCAFCIsConnected}), $h$ is surjective and then $h $ must be an
admissible epic. Moreover, the kernel $X:=\ker(h)$ exists in the category, and
is itself compact. We obtain the commutative diagram%
\[%
\xymatrix{
X \ar@{^{(}->}[r] \ar[d] & V \ar@{->>}[r]^{h} \ar[d] & C \ar[d]^{\operatorname
{id}} \\
F \ar@{^{(}->}[r] & G \ar@{->>}[r] & C,
}%
\]
which shows that $\mathsf{LCA}_{F,C}$ is left special. Being left special and
left filtering, our claim is proven.
\end{proof}

While the full subcategory $\mathsf{LCA}_{F,C}$ is left $s$-filtering in
$\mathsf{LCA}_{F,qab}$ by the previous result, this is not true in
$\mathsf{LCA}_{F}$ or $\mathsf{LCA}_{F,qab+fd}$.

\begin{example}
\label{Example_LCAFCNotLeftSpecial}While $\mathsf{LCA}_{F,C}$ remains left
filtering in $\mathsf{LCA}_{F}$, it is not left special in $\mathsf{LCA}_{F}$.
To see this, consider the diagram%
\[%
\xymatrix{
& V \ar@{->>}[r] \ar[d]_{w} & {\mathbb{A}}/F \ar@{=}[d] \\
F \ar@{^{(}->}[r] & \mathbb{A} \ar@{->>}[r] & {\mathbb{A}}/F
}%
\]
The lower row is the ad\`{e}le sequence. The quotient $\mathbb{A}/F\simeq
F^{\vee}$ is a compact $F$-vector space. Suppose $\mathsf{LCA}_{F,C}$ is left
special in $\mathsf{LCA}_{F}$. Then there exists an object $V\in
\mathsf{LCA}_{F,C}$ and an arrow $w$ making the above diagram commute. The
ad\`{e}les have the shape%
\[
\mathbb{A}\cong\left.  \prod\nolimits^{\prime}\right.  (\widehat{F}%
_{P},\widehat{\mathcal{O}}_{P})\oplus\bigoplus_{\sigma\in S}\mathbb{R}%
_{\sigma}\text{,}%
\]
where $S$ is the list of all real and complex places of $F$, $P$ runs through
all finite places, $\widehat{F}_{P}$ denotes the local field at $P$, and
$\widehat{\mathcal{O}}_{P}$ its ring of integers. Since $V$ is connected
(Lemma \ref{lemma_LCAFCIsConnected}), so is its image $w(V)$ inside
$\mathbb{A}$. The projection of the image to the finite ad\`{e}les, which are
topological torsion and thus totally disconnected, must be trivial. Hence, $w
$ factors as $w:V\overset{v}{\rightarrow}\bigoplus_{\sigma\in S}%
\mathbb{R}_{\sigma}\hookrightarrow\mathbb{A}$. However, $V$ is also compact,
so the image of $v$ in the vector $\mathcal{O}$-module is compact. However, a
real vector space has no non-trivial compact subgroups. We deduce that $w$ is
the zero map. Contradiction.
\end{example}

Let us briefly recall an important exact sequence in $\mathsf{LCA}%
_{\mathcal{O}}$.

\begin{definition}
Let $S$ denote the set of all infinite places of $F$. If $J$ denotes any
non-zero ideal of $\mathcal{O}$, there is a short exact sequence%
\begin{equation}
J\hookrightarrow\bigoplus_{\sigma\in S}\mathbb{R}_{\sigma}\twoheadrightarrow
\mathbb{T}_{J}\text{,}\label{lMinkSeq}%
\end{equation}
where the first arrow is plainly the inclusion along the Minkowski embedding,
and $\mathbb{T}_{J}$ is our preferred notation for the respective quotient. We
call it the \emph{Minkowski sequence} of $J$.
\end{definition}

For $J=(1)$, we recover the usual Minkowski embedding of the ring of integers.
The notation $\mathbb{T}_{J}$ is supposed to stress that the underlying
additive group of $\mathbb{T}_{J}$ is indeed a torus.

We come to the central result of the entire section.

\begin{theorem}
\label{thm_qabfdIsLeftSpecial}The full subcategory $\mathsf{LCA}_{F,qab+fd}$
is left $s$-filtering in $\mathsf{LCA}_{F}$.
\end{theorem}

\begin{proof}
Left filtering was already proven as part of Proposition
\ref{prop_LCAFqab_leftfilt}, so it only remains to verify that it is also left
special.\newline\textit{(Step 1) }Let $G\in\mathsf{LCA}_{F}$ be arbitrary.
Suppose for the moment that $H$ is a vector-free object in $\mathsf{LCA}%
_{F,qab}$. Let $f:G\twoheadrightarrow H$ be given. We have a presentation
$H=\bigcup_{n\geq1}\frac{1}{n}C$ with $C$ a compact clopen $\mathcal{O}%
$-submodule of $H$. Consider $f^{-1}(C)\subseteq G$. As $C$ is open in $H$,
$f^{-1}(C)$ is open in $G$. As $f$ is an admissible epic, it is an open map,
and thus the restriction $\tilde{f}:f^{-1}(C)\rightarrow H$ is an open map as
well. Moreover, it is clearly surjective and hence $\tilde{f}$ is an
admissible epic in $\mathsf{LCA}_{\mathcal{O}}$. Using \cite[Theorem
2.15]{obloc} we may choose a splitting $f^{-1}(C)\cong\bigoplus_{\nu\in
I^{\prime}}\mathbb{R}_{\nu}\oplus W$ such that $W\in\mathsf{LCA}_{\mathcal{O}%
}$ has no vector $\mathcal{O}$-module summand, and further we obtain a certain
commutative diagram%
\begin{equation}%
\xymatrix{
& \tilde{C} \ar@{^{(}->}[d]_{i} \ar@{->}[dr]^{h} \\
f^{-1}(C) \ar@{=}[r] & {\bigoplus_{\nu\in I^{\prime}}\mathbb{R}_{\nu}} \oplus
W \ar@{->>}[r]_-{\tilde{f} } \ar@{->>}[d] & C \\
& {\bigoplus_{\nu\in I^{\prime}}\mathbb{R}_{\nu}} \oplus D.
}%
\label{ld9}%
\end{equation}
Here $\tilde{C}$ is a compact clopen in $W$, and $D$ its discrete quotient.
Since $\tilde{C}$ is compact, its set-theoretic image $\operatorname*{im}(h)$
is compact and thus we obtain a new diagram%
\begin{equation}%
\xymatrix{
\tilde{C} \ar@{^{(}->}[d]_{i} \ar[dr]^{0} \\
f^{-1}(C) \ar@{->>}[r]_-{\tilde{f} } \ar@{->>}[d] & C/{\operatorname{im}(h)}
\\
{\bigoplus_{\nu\in I^{\prime}}\mathbb{R}_{\nu}} \oplus D \ar@{-->>}[ur]_{w},
}%
\label{lFig_Quot}%
\end{equation}
where $w$ exists by the universal property of cokernels, and $w$ is an
admissible epic by \cite[Prop. 7.6]{MR2606234} (which uses that $\mathsf{LCA}%
_{F}$ is quasi-abelian and thus idempotent complete). By the classification of
admissible quotients of vector groups in $\mathsf{LCA}$, \cite[Corollary 2 to
Theorem 7]{MR0442141},\ it follows that the underlying additive group of
$C/\operatorname*{im}(h)$ is isomorphic to $\mathbb{R}^{a}\oplus\mathbb{T}%
^{b}\oplus D^{\prime}$ for suitable finite $a,b$ and $D^{\prime}$ discrete. On
the other hand, $C$ is a compact $\mathcal{O}$-module, and so must be all its
quotients, and then we must necessarily have $a=0$ and $D^{\prime}$ finite.
Hence, $C/\operatorname*{im}(h)$ has no small subgroups as an object in
$\mathsf{LCA}_{\mathbb{Z},nss}$ and by \cite[Proposition 2.13 (1)]{obloc} we
can only have%
\begin{equation}
C/\operatorname*{im}(h)\cong\underline{\bigoplus_{J\in\mathcal{I}}%
\mathbb{T}_{J}}\oplus D^{\prime}\label{ld6}%
\end{equation}
for a suitable \textit{finite} list $\mathcal{I}$ of ideals $J$ (the underline
plays no r\^{o}le yet, we shall use it below). We define a morphism%
\begin{equation}
\pi:V:=\bigoplus_{J\in\mathcal{I}}\left(  \bigoplus_{\mu\in S}\mathbb{R}_{\mu
}\right)  \twoheadrightarrow\underline{\bigoplus_{J\in\mathcal{I}}%
\mathbb{T}_{J}}\hookrightarrow C/\operatorname*{im}(h)\text{,}\label{ld7}%
\end{equation}
where $S$ denotes the finite list of all real and complex places of $F$ as
follows: The map $\pi$ is the identity on the $\mu$-summands, and the dualized
Minkowski embedding on the $J$-summands (as in Equation \ref{lMinkSeq}):%
\[%
\xymatrix{
V \ar[dr]^{\pi} \\
f^{-1}(C) \ar@{->>}[r]_-{\tilde{f} } & C/{\operatorname{im}(h)}. \\
}%
\]
As $V$ is a vector $\mathcal{O}$-module, it is a projective object in
$\mathsf{LCA}_{\mathcal{O}}$ by \cite[Theorem 2.17]{obloc} and thus $\pi$ can
be lifted along the admissible epic $\tilde{f}$ in\ Diagram \ref{lFig_Quot}.
We call this lift $\widehat{\pi}$. Joint with $i:\tilde{C}\hookrightarrow
f^{-1}(C)$ of that diagram, we obtain the new diagram%
\[%
\xymatrix{
V \oplus\tilde{C} \ar[d]_{\hat{\pi} + i} \ar[dr]^{\pi} \\
f^{-1}(C) \ar@{->>}[r]_-{\tilde{f} } & C/{\operatorname{im}(h)}. \\
}%
\]
Define%
\begin{equation}
\widehat{C}:=\bigcup_{n\geq1}\frac{1}{n}\tilde{C}\label{ld10}%
\end{equation}
in $G$. As $\tilde{C}$ is an open $\mathcal{O}$-submodule of $f^{-1}(C)$, and
the latter is open in $G$, $\tilde{C}$ is open in $G$, so $\widehat{C}$ is a
union of opens and therefore a clopen $\mathcal{O}$-submodule of $G$.
Moreover, since all $\alpha^{-1}$ for $\alpha\in\mathcal{O}\setminus\{0\} $
can be written as $\beta/n$ for $\beta\in\mathcal{O}$ and $n\in\mathbb{Z}%
_{\geq1}$, it follows that $\widehat{C}$ is actually an open $F$-submodule of
$G$. Define a morphism%
\[
f\hat{\pi}\oplus f:V\oplus\widehat{C}\longrightarrow H
\]
as follows: On $V$ we map via $\hat{\pi}$ to $f^{-1}(C)$, which lies in $G$,
so applying $f$ makes sense. We could just as well have written $\tilde{f}%
\hat{\pi}$. Next, $\widehat{C}$ is defined as a subgroup of $G$, so $f$ is
also defined. We claim that $f\hat{\pi}\oplus f$ is a surjective continuous
$F$-linear map: Linearity and continuity are clear: It is composed from
continuous $\mathcal{O}$-module homomorphisms, and all modules in it are
$\mathcal{O}$-divisible. It remains to check surjectivity:\ Let $x\in H$ be
given. Then there exists some $n\in\mathbb{Z}_{\geq1}$ such that $nx\in C$.
Consider its image in $C/\operatorname*{im}(h)$. As $D^{\prime}$ is finite,
there exists some $N\in\mathbb{Z}_{\geq1}$ such that $Nnx$ lies in the
underlined summands in Equation \ref{ld6} (e.g., the cardinality
$N:=\#D^{\prime}$ will do). Since $\pi$ in Equation \ref{ld7} surjects onto
this summand, and since $\hat{\pi}$ is a lift of $\pi$ along $\tilde{f}$,
there exists some $v\in V$ with $\tilde{f}\hat{\pi}v\equiv Nnx$ in
$C/\operatorname*{im}(h)$. Thus,%
\[
f\hat{\pi}v-Nnx\in\operatorname*{im}(h)\text{.}%
\]
We return to Diagram \ref{ld9}. We obtain that there exists some $\tilde{c}%
\in\tilde{C}$ with $f\hat{\pi}v-Nnx=h(\tilde{c})$. Hence, $x=\hat{\pi}\left(
\frac{1}{Nn}v\right)  -\frac{1}{Nn}\tilde{c}$. Here $\frac{1}{Nn}v\in V $ as
$V$ is an $F$-vector space, and $\frac{1}{Nn}\tilde{c}\in\widehat{C}$ by
Equation \ref{ld10}. This proves surjectivity. Next, we claim that $f\hat{\pi
}\oplus f$ is an open map: This can be checked on the level of the underlying
additive groups in $\mathsf{LCA}$. But $V$ is a vector group, so
$V\oplus\widehat{C}$ is $\sigma$-compact. By Pontryagin's Open Mapping Theorem
\cite[Theorem 3]{MR0442141} this implies that $f\hat{\pi}\oplus f$ is open.
Hence, $f\hat{\pi}\oplus f$ is an admissible epic in $\mathsf{LCA}_{F}$. Thus,
defining%
\[
A:=\ker\left(  f\hat{\pi}+f:V\oplus\widehat{C}\twoheadrightarrow H\right)
\qquad\text{in}\qquad\mathsf{LCA}_{F}\text{,}%
\]
we get%
\begin{equation}%
\xymatrix{
A \ar@{^{(}->}[r] \ar@{-->}[d] & V\oplus\widehat{C} \ar@{->>}[r]^-{f \hat{\pi}
+ f} \ar[d]^{\psi} & H \ar@{=}[d] \\
\operatorname{ker}(f) \ar@{^{(}->}[r] & G \ar@{->>}[r]_-{f} & H,
}%
\label{lcx}%
\end{equation}
where $\psi$ is the sum of $V\overset{\hat{\pi}}{\rightarrow}f^{-1}%
(C)\hookrightarrow G$ and the inclusion of $\widehat{C}$ into $G$ as defined
in Equation \ref{ld10}. We point out that the entire upper row lies in
$\mathsf{LCA}_{F,qab+fd}$: The terms $V\oplus\widehat{C}$ and $H$ lie in
$\mathsf{LCA}_{F,qab}$, and for the kernel $A$ this follows from the Kernel
Theorem (Theorem \ref{thm_KernelOfqabfds}).\newline\textit{(Step 2) }So far we
had assumed that $H$ is a vector-free object in $\mathsf{LCA}_{F,qab}$.
However, in order to prove our claim, we really need to handle the general
case of $H\in\mathsf{LCA}_{F,qab+fd}$. This means that%
\[
H\cong H_{0}\oplus\bigoplus_{\gamma\in I}\mathbb{R}_{\gamma}\oplus D\text{,}%
\]
where $I$ is a finite list of real and complex places, $D$ is a
finite-dimensional discrete $F$-vector space, and $H_{0}$ is a vector-free
object in $\mathsf{LCA}_{F,qab}$. Next, observe that each $\mathbb{R}_{\gamma
}$ as well as any discrete vector space is a projective object in
$\mathsf{LCA}_{F}$ by Lemma \ref{Lem_InjProjInLCAF}. Thus, under the
admissible epic%
\[
f:G\twoheadrightarrow H\text{,}%
\]
the direct summands of $H$ possess a section to $G$. We may therefore rewrite
$f$ as%
\[
f:G^{\prime}\oplus\bigoplus_{\gamma\in I}\mathbb{R}_{\gamma}\oplus
D\twoheadrightarrow H_{0}\oplus\bigoplus_{\gamma\in I}\mathbb{R}_{\gamma
}\oplus D\text{,}%
\]
with $f$ being induced from the identity on the part $\bigoplus_{\gamma\in
I}\mathbb{R}_{\gamma}\oplus D$ and some $\tilde{f}:G^{\prime}%
\twoheadrightarrow H_{0}$ on the remaining direct summand. Thus, we may apply
Step 1 of the proof to the morphism $\tilde{f}$, resulting in a Diagram of the
shape of Diagram \ref{lcx} for $\tilde{f}$, and then compatibly add the
further projective direct summands. For example, for a summand of the shape
$\mathbb{R}_{\gamma}$, the section allows us to augment Diagram \ref{lcx} to%
\[%
\xymatrix{
A \ar@{^{(}->}[r] \ar@{-->}[d] & V\oplus\widehat{C} {\oplus{\mathbb{R}%
_{\gamma}}} \ar@{->>}[r]^-{f \hat{\pi} + f} \ar[d]^{\psi} & H_0 {\oplus
{\mathbb{R}_{\gamma}}} \ar@{=}[d] \\
\operatorname{ker}(f) \ar@{^{(}->}[r] & G^{\prime} {\oplus{\mathbb{R}_{\gamma
}}} \ar@{->>}[r]_-{f} & H_0 {\oplus{\mathbb{R}_{\gamma}}},
}%
\]
Since adding $\bigoplus_{\gamma\in I}\mathbb{R}_{\gamma}\oplus D$ to an object
in $\mathsf{LCA}_{F,qab+fd}$ yields an object in $\mathsf{LCA}_{F,qab+fd}$
(recall that $D$ is finite-dimensional!), the upper row still lies entirely in
$\mathsf{LCA}_{F,qab+fd}$. This finishes the proof.
\end{proof}

\begin{example}
The full subcategory $\mathsf{LCA}_{F,qab}$ is left filtering, but not left
special in $\mathsf{LCA}_{F}$. This follows from a variation of previous
counter-examples: Observe that in the ad\`{e}le sequence $F\hookrightarrow
\mathbb{A}\twoheadrightarrow\mathbb{A}/F$ the middle term $\mathbb{A}$ is
quasi-adelic. We already know that $\mathsf{LCA}_{F,qab}$ is left filtering in
$\mathsf{LCA}_{F}$ by Proposition \ref{prop_LCAFqab_leftfilt}. Hence, if
$\mathsf{LCA}_{F,qab}$ was left special, then it would be closed under
subobjects (by \cite[Proposition A.2]{MR3510209} and the resulting property
\cite[Definition A.1, (1)]{MR3510209}). Thus, we could conclude that $F$ was
quasi-adelic, which by Lemma \ref{lemma_nomapsfromqabtodiscrete} (applied to
the identity map) is clearly false.
\end{example}

\begin{theorem}
Let $G^{\prime}\hookrightarrow G\twoheadrightarrow G^{\prime\prime}$ be an
exact sequence in $\mathsf{LCA}_{F}$. Then $G\in\mathsf{LCA}_{F,qab+fd}$ if
and only if $G^{\prime}$ and $G^{\prime\prime}$ both lie in $\mathsf{LCA}%
_{F,qab+fd}$.
\end{theorem}

\begin{proof}
Our claim is the property \cite[Definition A.1, (1)]{MR3510209} and by
\cite[Proposition A.2]{MR3510209} it follows from being left $s$-filtering.
\end{proof}

\section{A theorem of Braconnier--Vilenkin type, and $P$-adic finiteness}

Let $F$ be a number field and $\mathcal{O}$ its ring of integers. For every
place $P$, finite or infinite, write $\widehat{F}_{P}$ for the local field at
$P$, so this will be a finite extension of some $\mathbb{Q}_{p}$, or of
$\mathbb{R}$, depending on whether the place is finite or infinite. If finite,
we denote by $\widehat{\mathcal{O}}_{P}$ the ring of integers in $\widehat
{F}_{P}$.

\begin{definition}
[Restricted product]Let $I$ be an index set. Let $(J_{i})_{i\in I}$ be objects
in $\mathsf{LCA}$, and $K_{i}\subseteq J_{i}$ a clopen compact subgroup for
each $i\in I$. We write $\left.  \prod\nolimits_{i\in I}^{\prime}\right.
(J_{i},K_{i})$ to denote the \emph{restricted product} over $I$, topologized
relative to the subgroups $K_{i}\subseteq J_{i}$. In the group-theoretical
literature, the same concept is also known as the `\emph{local direct
product}'. See \cite[P.14]{MR637201}.
\end{definition}

Given this datum, the restricted product is always an object in $\mathsf{LCA}
$, and it always has the product of the compacta $\prod_{i\in I}K_{i}%
\subseteq\left.  \prod\nolimits_{i\in I}^{\prime}\right.  (J_{i},K_{i})$ as a
compact clopen subgroup.

We recall the classical Braconnier--Vilenkin theorem.

\begin{theorem}
[Braconnier--Vilenkin, original version]If $G\in\mathsf{LCA}$ is a topological
torsion group and $C\hookrightarrow G $ an arbitrary compact clopen subgroup,
then there exists an isomorphism to a restricted product%
\[
G\cong\left.  \prod\nolimits^{\prime}\right.  (G^{p},G^{p}\cap C)\text{,}%
\]
where the restricted product runs over all prime numbers $p$ and $(-)^{p}$
denotes the topological $p$-torsion elements.
\end{theorem}

See \cite[(3.13) Theorem]{MR637201}. Levin has extended this type of result to
objects in $\mathsf{LCA}_{\mathcal{O}}$. Firstly, he generalized the concept
of topological $p$-torsion elements to prime ideals:

\begin{definition}
[{Levin \cite[\S 3]{MR0310125}}]\label{def_TopTorsALaLevin}Suppose
$G\in\mathsf{LCA}_{\mathcal{O}}$. Let $P$ be a non-zero prime ideal of
$\mathcal{O}$. We say that $g\in G$ is\emph{\ topological }$P$\emph{-torsion}
if for every neighbourhood of zero, $0\in U\subset G$, there exists some
$N_{U}\geq1$ such that we have the inclusion of sets%
\begin{equation}
P^{N_{U}}\cdot g\subseteq U\text{.}\label{lqa1}%
\end{equation}

\begin{enumerate}
\item We write $G^{P}$ to denote the subset of topological $P$-torsion
elements in $G$.

\item $G$ is a \emph{topological }$P$\emph{-torsion} group if all its elements
are topological $P$-torsion.
\end{enumerate}
\end{definition}

\begin{proposition}
\label{prop_TopTorsionIsVectorFreeAdelicBlock}Given $G\in\mathsf{LCA}_{F}$, it
is topologically torsion in the sense of Definition \ref{def_TopTors} if and
only if it is a vector-free adelic block.
\end{proposition}

\begin{proof}
By Lemma \ref{lemma_StructVectorFreeAdelicBlock} all vector-free adelic blocks
are topologically torsion, so we only need to prove the converse. Let
$G\in\mathsf{LCA}_{F}$ be topologically torsion. We apply Theorem
\ref{thm_StructLCAF} and obtain a decomposition $Q\hookrightarrow
G\twoheadrightarrow D$, with $Q$ quasi-adelic and $D$ a discrete $F$-vector
space. Following Definition \ref{def_TopTors} the $F$-vector space $G$ is the
union of its compact $\mathcal{O}$-submodules, and thus the same is true for
$D$. However, as $D$ is discrete, its compact subsets are finite, thus must be
torsion $\mathcal{O}$-modules, which is impossible inside an $F$-vector space,
unless they are all zero. We conclude that $D=0$. Hence, $G=Q$ is
quasi-adelic, and $Q$ is topologically torsion. Next, we apply Proposition
\ref{prop_DecomposeQuasiAdelicBlock} and get a decomposition $\bigoplus
_{\sigma\in I}\mathbb{R}_{\sigma}\oplus V\hookrightarrow Q\twoheadrightarrow
A$, where $I$ is a finite list of real and complex places, $V$ a compact
$F$-vector space and $A$ a vector-free adelic block. Since $V$ is compact, it
is connected (Lemma \ref{lemma_LCAFCIsConnected}), so the entire left hand
side term in the above sequence is connected. However, $Q$ is topologically
torsion, so by Robertson \cite[(3.15)\ Theorem]{MR0217211}, $Q$ is totally
disconnected. This forces the entire left hand side term to be zero. Hence,
$Q=A$, proving the claim.
\end{proof}

\begin{corollary}
\label{Cor_AdelicBlocksAreClosedUnderDuality}The Pontryagin dual of an adelic
block is an adelic block.
\end{corollary}

\begin{proof}
If $G$ is an adelic block, it has a presentation%
\[
G=\bigoplus_{\sigma\in I}\mathbb{R}_{\sigma}\oplus H\qquad\text{with}\qquad
H=\bigcup_{n\geq1}\frac{1}{n}C\text{,}%
\]
where $C$ is a compact clopen $\mathcal{O}$-module in $G$ with $\bigcap
_{n\geq1}nC=0$, and $I$ is a finite list of real and complex places. In
particular, $H$ is a vector-free adelic block. By Lemma
\ref{lemma_StructVectorFreeAdelicBlock} the dual $H^{\vee}$ is topologically
torsion and it follows from Proposition
\ref{prop_TopTorsionIsVectorFreeAdelicBlock} that $H^{\vee}$ is a vector-free
adelic block. Thus, since finite direct sums commute with Pontryagin duality,
we obtain%
\[
G^{\vee}\simeq\bigoplus_{\sigma\in I}\mathbb{R}_{\sigma}^{\vee}\oplus H^{\vee}%
\]
and since $\mathbb{R}_{\sigma}^{\vee}$ is (non-canonically) self-dual, we
deduce that $G^{\vee}$ is an adelic block.
\end{proof}

The following is a variation of the Braconnier--Vilenkin Theorem in the
context of $\mathsf{LCA}_{F}$. Levin has already given such a result in
\cite[Theorem 3]{MR0310125} for $\mathsf{LCA}_{\mathcal{O}}$, but our
variation for $\mathsf{LCA}_{F}$ differs by a crucial additional finiteness property:

\begin{theorem}
[Braconnier--Vilenkin type result]\label{thm_StrongBracVilenkin}Every object
$G\in\mathsf{LCA}_{F,ab}$ is isomorphic to a restricted product%
\[
G\cong\bigoplus_{\sigma\in I}\mathbb{R}_{\sigma}\oplus\left.  \prod
\nolimits^{\prime}\right.  (G^{P},G^{P}\cap C)\text{,}%
\]
where $I$ is a finite list of real and complex places, $P$ runs over all
finite places of $F$, and $C$ is an arbitrary compact clopen $\mathcal{O}%
$-submodule of $G$. Moreover,

\begin{enumerate}
\item each $G^{P}$ is a topological $P$-torsion group;

\item $G^{P}\hookrightarrow G$ is an admissible monic in $\mathsf{LCA}_{F,ab}
$;

\item $G^{P}\cap C$ is a compact clopen $\mathcal{O}$-submodule of $G^{P}$;

\item each $G^{P}$ is a finite-dimensional $\widehat{F}_{P}$-vector space,
where the vector space structure is canonically given by $P$-adically
completing the $F$-vector space structure of $G^{P}$;

\item In particular, $G^{P}\cap C$ is a full rank $\widehat{\mathcal{O}}_{P}%
$-lattice inside $G^{P}$;

\item $G$ is canonically a topological $\mathbb{A}$-module, where $\mathbb{A}
$ denotes the ad\`{e}les of $F$ as a locally compact ring;

\item $G$ is (non-canonically) isomorphic to its dual $G^{\vee}$.
\end{enumerate}
\end{theorem}

\begin{remark}
Property (7) generalizes the self-duality of the ad\`{e}les of a number field.
\end{remark}

We stress that the finite dimension of each $G^{P}$ is absolutely crucial for
the remaining computations. If $\sigma^{\prime}$ denotes an infinite place of
$F$, let us agree to write%
\begin{equation}
G^{\sigma^{\prime}}:=\bigoplus_{\sigma\in I\text{ with }\sigma=\sigma^{\prime
}}\mathbb{R}_{\sigma}\text{,}\label{lqa4}%
\end{equation}
just so that we have a uniform notation, resembling how we can isolate the
topological $P$-torsion parts in the above decomposition.

\begin{proof}
\textit{(Step 1)} Firstly, using Proposition
\ref{prop_DecomposeQuasiAdelicBlock} we split off the vector $\mathcal{O}%
$-module part of $G$, so that we have $G\cong\bigoplus_{\sigma\in I}%
\mathbb{R}_{\sigma}\oplus V$, where $V$ is a vector-free adelic block. We must
have $C\subseteq V$ since the image of $C$ under the projection to the vector
$\mathcal{O}$-module summand must still be compact, and thus must be zero.
Thus, we prove our claim if we prove it for all $G$ in the special situation
of $G=V$.\newline\textit{(Step 2)} We first intepret such a vector-free part
$V$ as an object in $\mathsf{LCA}_{\mathcal{O}}$. By a slight variation of
\cite[(3.8) Lemma]{MR637201} it follows that each $V^{P}$ is closed in $V$,
and thus itself an object in $\mathsf{LCA}_{\mathcal{O}}$ at all; in
particular its subspace topology is indeed locally compact. This already
proves Claims (1) and (2). Now apply Levin's version of the
Braconnier--Vilenkin theorem, \cite[Theorem 3]{MR0310125}, noting that
$M_{0}=0$, $R=0$, $M_{1}=V$, so $M_{3}\cong V$ in the notation \textit{loc.
cit}. This yields an isomorphism in $\mathsf{LCA}_{\mathcal{O}}$, namely for
any compact clopen $\mathcal{O}$-submodule $C\hookrightarrow V$ we get an
isomorphism%
\begin{equation}
V\cong\left.  \prod\nolimits_{P\in\left.  \text{finite places}\right.
}^{\prime}\right.  (V^{P},V^{P}\cap C)\label{lqa3}%
\end{equation}
in $\mathsf{LCA}_{\mathcal{O}}$. We remark that $V^{P}\cap C$ is compact. As
$C$ is open in $V$, $V^{P}\cap C$ is open in $V^{P}$, and thus compact clopen
in $V^{P}$. This shows that the restricted product on the right hand side
makes sense, and also confirms Claim (3). Next, we need to argue that the
right hand side of Equation \ref{lqa3} lies not just in $\mathsf{LCA}$, but in
$\mathsf{LCA}_{F}$: We claim that $V^{P}$ is an $\mathcal{O}$-submodule of
$V$. This is easy: By definition for every element $x\in V^{P}$ and every
neighbourhood of zero, $0\in U\subset V^{P}$, there exists some $N_{U}\geq1$
such that%
\begin{equation}
P^{N_{U}}\cdot x\subseteq U\text{.}\label{lqa2}%
\end{equation}
Since $P$ and thus $P^{N_{U}}$ is an ideal in $\mathcal{O}$, for any
$\alpha\in\mathcal{O}$ we have $P^{N_{U}}\cdot(\alpha\cdot x)=(P^{N_{U}}%
\cdot\alpha)\cdot x\subseteq P^{N_{U}}\cdot x\subseteq U$, so the defining
property of $V^{P}$ still holds for all $\alpha x$ with $\alpha\in\mathcal{O}%
$. Further, we claim that $V^{P}$ is even an $\widehat{\mathcal{O}}_{P}%
$-module. To this end note that the $\mathcal{O}$-module action gives us a
scalar multiplication%
\[
\mathcal{O}\times V^{P}\longrightarrow V^{P}\text{,}%
\]
but since $V^{P}$ is topological $P$-torsion, this canonically extends to a
scalar multiplication%
\begin{equation}
\underset{n\geq1}{\underleftarrow{\lim}}\,\mathcal{O}/P^{n}\times
V^{P}\longrightarrow V^{P}\label{lqa2a}%
\end{equation}
by employing the continuity condition of Equation \ref{lqa2}. However, the
projective limit of rings on the left agrees with $\widehat{\mathcal{O}}_{P}$.
Next, we claim that $V^{P}$ is also closed under division by $p$. Using the
valuation of the discrete valuation ring $\widehat{\mathcal{O}}_{P}$, we just
need to observe that for the prime number $p$ such that $(p)=P\cap\mathbb{Z}$
we have\footnote{We would not have to add $[F:\mathbb{Q}]$. It would be
sufficient to add $e$, where $e$ is the ramification index of $p$ in the
extension $F/\mathbb{Q}$.}%
\[
P^{N_{U}+[F:\mathbb{Q}]}\cdot\frac{x}{p}\subseteq P^{N_{U}}\cdot x\subseteq
U\text{.}%
\]
However, combining divisibility by $p$ and the $\widehat{\mathcal{O}}_{P}%
$-module structure, we learn that $V^{P}$ is canonically an $\widehat
{\mathcal{O}}_{P}[\frac{1}{p}]$-module, but $\widehat{\mathcal{O}}_{P}%
[\frac{1}{p}]=\widehat{F}_{P}$ is already the field of fractions. Thus, each
$V^{P}$ is a $\widehat{F}_{P}$-vector space in a canonical fashion, and the
$F$-scalar multiplication of elements of $V^{P}$ inside $V$ is compatible with
the $\widehat{F}_{P}$-scalar multiplication under the canonical inclusion
$F\subseteq\widehat{F}_{P}$. This proves Claim (4), except for the
finite-dimensionality.\newline\textit{(Step 3) }The previous step has
constructed the isomorphism of Equation \ref{lqa3} in the category
$\mathsf{LCA}_{\mathcal{O}}$. We lift it to $\mathsf{LCA}_{F}$: As we have
seen in the previous step, the right hand side is not just an object in
$\mathsf{LCA}_{\mathcal{O}}$, but is algebraically an $F$-vector space. Since
we already have an isomorphism of the level of the underlying additive groups,
it follows that the right hand side is also $\sigma$-compact. Now any
$\alpha\in\mathcal{O}-\{0\}$ induces a surjective continuous endomorphism,
which by Pontryagin's Open Mapping Theorem must be an open map and thus a
homeomorphism. It follows that the right hand side is also a topological
$F$-vector space and thus an object in $\mathsf{LCA}_{F}$. Now Lemma
\ref{lemma_LCAFtoLCAO} implies that our isomorphism in $\mathsf{LCA}%
_{\mathcal{O}}$ even gives an isomorphism in $\mathsf{LCA}_{F}$. Since
$V\in\mathsf{LCA}_{F}$, and the category is closed under Pontryagin duality,
the dual $V^{\vee}$ also lies in $\mathsf{LCA}_{F}$, and thus both $V$ and
$V^{\vee}$ are divisible as locally compact abelian groups.\ One also says
that $V$ is \textquotedblleft bi-divisible\textquotedblright. By a result of
Sheng L. Wu \cite[Lemma 6.12, (3)$\Rightarrow$(1)]{MR1173767}, it follows that
in $\mathsf{LCA}$ there exists an isomorphism to a restricted product%
\[
V\cong\left.  \prod\nolimits_{p\text{ prime number}}^{\prime}\right.
(\mathbb{Q}_{p}^{m_{p}},\mathbb{Z}_{p}^{m_{p}})\qquad\text{in}\qquad
\mathsf{LCA}%
\]
for some $m_{p}\in\mathbb{Z}_{\geq0}$. So, in particular, each $m_{p}$ is
\textit{finite}. Since, under the forgetful functor $\mathsf{LCA}%
_{F}\rightarrow\mathsf{LCA}$, we must have%
\[
V^{P}\subseteq(\mathbb{Q}_{p}^{m_{p}},\mathbb{Z}_{p}^{m_{p}})\text{,}%
\]
and under $\mathbb{Q}_{p}\subseteq\widehat{F}_{P}$ the vector space structures
are compatible, we conclude that the $\widehat{F}_{P}$-vector space $V^{P}$
must also be finite-dimensional (as it is even finite-dimensional over the
deeper base $\mathbb{Q}_{p}$). This proves Claim (4) in full. As both $V^{P}$
and $C$ are $\mathcal{O}$-modules, so is $V^{P}\cap C$. The same argument as
around Equation \ref{lqa2a} proves that $V^{P}\cap C$ is even a topological
$\widehat{\mathcal{O}}_{P}$-module, and since it is contained in the
finite-dimensional $\widehat{F}_{P}$-vector space $V^{P}$ it can only be
finitely generated free. As $V^{P}\cap C$ is clopen in $V^{P}$, and thus has
discrete quotient, it can only be a full rank lattice. This proves Claim (5).
Claim (6) follows from combining the respective $\widehat{F}_{P}$-vector space
structures for all places. Claim (7) follows, on the level of the underlying
locally compact additive group, from Wu's result that bi-divisibility implies
self-duality \cite[Lemma 6.12, (3)$\Rightarrow$(2)]{MR1173767}. It can be
shown by a direct argument that this self-duality holds in $\mathsf{LCA}_{F}$
and not just in $\mathsf{LCA}$. We leave this to the reader, because a more
elegant proof is to wait since we prove Proposition \ref{prop_ProductCatLCAS}
and Lemma \ref{lemma_psiInftyIsEquivalence} (whose proofs do not rely on
self-duality) and which give an exact equivalence of the category
$\mathsf{LCA}_{F,ab}$ with an alternative categorical description where the
self-duality is evident).
\end{proof}

We need to comment on how arbitrary the choice of $C$ in the previous theorem was.

\begin{lemma}
\label{lemma_LatticeSys1}Suppose $G\in\mathsf{LCA}_{F,ab}$ is vector-free, say%
\begin{equation}
G=\bigcup_{n\geq1}\frac{1}{n}C\label{lqa5b}%
\end{equation}
for $C$ a compact clopen $\mathcal{O}$-submodule, then every compact
$\mathcal{O}$-submodule $\tilde{C}\subseteq G$ is contained in $\frac{1}{m}C$
for a suitable $m\geq1$.
\end{lemma}

\begin{proof}
Immediate: $\tilde{C}$ is compact and Equation \ref{lqa5b} defines an open
cover whose opens are arranged as a filtering poset. In order to cover the
compactum $\tilde{C}$, finitely many opens suffice. By the filtering property,
we find a single open which is sufficiently big.
\end{proof}

\begin{lemma}
\label{lemma_LatticeSys2}Suppose $G\in\mathsf{LCA}_{F,ab}$ and $C,C^{\prime}$
are compact clopen $\mathcal{O}$-submodules. Then $C^{P}=C^{\prime P}$ for all
finite places with at most finitely many exceptions.
\end{lemma}

\begin{proof}
We may assume that $G$ is vector-free, as before. \textit{(Step 1)} Say
$G=\bigcup_{n\geq1}\frac{1}{n}\tilde{C}$ for $\tilde{C}$ a compact clopen
$\mathcal{O}$-submodule. By the previous lemma, $C\subseteq\frac{1}{n}%
\tilde{C}$ for some $n$. As $C$ is open in $G$, it is open in $\frac{1}%
{n}\tilde{C}$ and thus the quotient $\frac{1}{n}\tilde{C}/C$ is discrete.
However, $\tilde{C}$ is compact, and so this quotient is also compact and
therefore finite. Let $N$ be its cardinality. Then for all $g\in\frac{1}%
{n}\tilde{C}$ we have $Ng\in C$. However, $N$ has only finitely many prime
factors, so for all but finitely many $P$, if $g\in\frac{1}{n}\tilde{C}^{P}$,
$N$ acts as unit in the $\widehat{\mathcal{O}}_{P}$-module structure. Hence,
except for finitely many $P$, we have $\left(  \frac{1}{n}\tilde{C}/C\right)
^{P}=0$ and therefore $(\frac{1}{n}\tilde{C})^{P}=C^{P}$. \textit{(Step 2)}
Now, for the given $C,C^{\prime}$ in our claim, pick $n$ large enough that we
jointly have $C,C^{\prime}\subseteq\frac{1}{n}\tilde{C}$ and apply Step 1 to
each of these inclusions separately. The set of places where possibly
$C^{P}\neq C^{\prime P}$ then is the union of the possible sets of exceptions
for both steps, but that is still a finite set.
\end{proof}

With these theorems available, we may now summarize all our structural results
as follows.

\begin{theorem}
[Principal Structure Theorem]\label{thm_PrincipalStructureTheoremForLCAF}Let
$G\in\mathsf{LCA}_{F}$ be arbitrary. Then there exists a direct sum
decomposition%
\[
G\cong V\oplus A\oplus D\text{,}%
\]
where $V$ is a compact $F$-vector space, $A$ is an adelic block, and $D$ is a
discrete $F$-vector space. Pontryagin duality exchanges the full subcategories
of compact vs. discrete $F$-vector spaces, while adelic blocks are closed
under Pontryagin duality. Even more strongly, all adelic blocks are self-dual,
i.e. $A\simeq A^{\vee}$ (non-canonically).
\end{theorem}

\begin{proof}
By Theorem \ref{thm_StructLCAF} we have $G\cong Q\oplus D$ with $Q$
quasi-adelic, and by Proposition \ref{prop_DecomposeQuasiAdelicBlock} we may
write $Q\cong V\oplus A$ as in the claim. Finally, by Theorem
\ref{thm_StrongBracVilenkin} all adelic blocks are self-dual individually as
objects. Hence, the entire category $\mathsf{LCA}_{F,ab}$ is closed under
Pontryagin duality. Alternatively, the last claim also follows from Corollary
\ref{Cor_AdelicBlocksAreClosedUnderDuality} of course.
\end{proof}

\section{$K$-theory of adelic blocks}

We recall that if $(\mathsf{C}_{i})_{i\in I}$ (for some index set $I$) are
exact categories, we can define a \emph{product exact category} $\prod_{i\in
I}\mathsf{C}_{i}$ as follows: Its objects are tuples $\underline{X}%
:=(X_{i})_{i\in I}$ with $X_{i}\in\mathsf{C}_{i}$, and its morphisms
$\underline{f}:\underline{X}\longrightarrow\underline{Y}$ are morphisms
$\underline{f}=(f_{i})_{i\in I}$ with $f_{i}\in\operatorname*{Hom}%
\nolimits_{\mathsf{C}_{i}}(X_{i},Y_{i})$. Finally, a kernel-cokernel sequence
$\underline{X^{\prime}}\longrightarrow\underline{X}\longrightarrow
\underline{X^{\prime\prime}}$ is defined to be exact if and only if all
levelwise sequences $X_{i}^{\prime}\hookrightarrow X_{i}\twoheadrightarrow
X_{i}^{\prime\prime}$ are exact with respect to the exact structure of
$\mathsf{C}_{i}$. We leave it to the reader to check that this defines an
exact structure on the category $\prod_{i\in I}\mathsf{C}_{i}$. Clearly all
the projection functors $p_{j}:\prod_{i\in I}\mathsf{C}_{i}\longrightarrow
\mathsf{C}_{j}$ (for $j\in I$) are exact.

We will now define a functor (at first only on objects)%
\begin{equation}
\Xi:\mathsf{LCA}_{F,ab}\longrightarrow\prod_{P\in\text{places}}\mathsf{Vect}%
_{fd}(\widehat{F}_{P})\text{,}\qquad\qquad V\longmapsto(V^{P})_{P}\label{lqa5}%
\end{equation}
where $P$ runs over all places of $F$, both finite and infinite (so if $P$ is
a finite place, $V^{P}$ is the topological $P$-torsion part as in Theorem
\ref{thm_StrongBracVilenkin}, and if $P$ is an infinite place, we follow the
convention of Equation \ref{lqa4}).\ Equation \ref{lqa5} describes a mapping
on objects. Note that it is only well-defined since we know that the vector
$\mathcal{O}$-module part in $V$ has only finitely many summands, and since
Theorem \ref{thm_StrongBracVilenkin} assures us that the topological
$P$-torsion part is indeed a finite-dimensional $\widehat{F}_{P}$-vector space.

Furthermore, observe that $\Xi$ is compatible under duality: For this it
suffices to observe the following agreement:%
\[%
\begin{array}
[c]{c}%
\operatorname*{Hom}\nolimits_{cts}(\widehat{F}_{P},\mathbb{T})\simeq
\widehat{F}_{P}\\
\qquad\\
\text{\textit{(Pontryagin dual)}}%
\end{array}
\qquad\qquad\text{versus}\qquad\qquad%
\begin{array}
[c]{c}%
\operatorname*{Hom}\nolimits_{\widehat{F}_{P}}(\widehat{F}_{P},\widehat{F}%
_{P})\simeq\widehat{F}_{P}\\
\qquad\\
\text{\textit{(linear dual)}}%
\end{array}
\text{,}%
\]
i.e. both duality functors do the same (despite their very different definitions!).

\begin{remark}
[Local self-duality]In many ways this is a very remarkable fact. It exploits
the Pontryagin self-duality of local fields. This fact plays a big r\^{o}le in
number theory, e.g. in Tate's thesis. Note that for other types of fields,
e.g. the number field $F$ itself, $\operatorname*{Hom}\nolimits_{cts}%
(F,\mathbb{T})$ and $\operatorname*{Hom}\nolimits_{F}(F,F)$ are drastically
different. The former is uncountable, while the latter is countable.
\end{remark}

We still need to define $\Xi$ on morphisms: Let $f:V\rightarrow V^{\prime}$ be
a morphism of objects in $\mathsf{LCA}_{F,ab}$, say%
\[
V\cong\bigoplus_{\sigma\in I}\mathbb{R}_{\sigma}\oplus\left.  \prod
\nolimits^{\prime}\right.  (V^{P},V^{P}\cap C)\qquad V^{\prime}\cong%
\bigoplus_{\sigma\in I^{\prime}}\mathbb{R}_{\sigma}\oplus\left.
\prod\nolimits^{\prime}\right.  (V^{\prime P},V^{\prime P}\cap C^{\prime
})\text{.}%
\]
Note that the vector $\mathcal{O}$-module parts are the connected components
of $V$ (resp. $V^{\prime}$) and thus $f$ maps the vector $\mathcal{O}$-module
summand of $V$ to the one of $V^{\prime}$. Compatibility with the $F $-vector
space structure implies that for all infinite places $P$ the $\widehat{F}_{P}%
$-vector space summand among these can only map to the $\widehat{F}_{P}%
$-vector space summand in $V^{\prime}$ for the same place $P$. Similarly,
using that all non-zero primes in $\mathcal{O}$ are pairwise coprime, we see
that $f$ can only map $V^{P}$ to $V^{\prime P}$ for the same prime $P$. Thus,
$f$ must actually be a diagonal morphism, we may (formally) write
$f=\sum_{P\in\text{places}}f^{P}$, and note that by projection and inclusion
each $f^{P}$ is uniquely determined, and defines a morphism%
\begin{equation}
f^{P}:V^{P}\longrightarrow V^{\prime P}\qquad\text{in}\qquad\mathsf{LCA}%
_{F}\text{.}\label{lqa6}%
\end{equation}
A continuity argument exactly as in the proof of Theorem
\ref{thm_StrongBracVilenkin} (cf. the argument around Equation \ref{lqa2a})
shows that each $f^{P}$ can be promoted to a continuous $\widehat{F}_{P}%
$-vector space morphism. Moreover, since we know that $V^{P}$, $V^{\prime P}$
are finite-dimensional $\widehat{F}_{P}$-vector spaces, $f^{P}$ indeed defines
a morphism in $\mathsf{Vect}_{fd}(\widehat{F}_{P})$. This completes the
definition of the functor $\Xi$.

Next, note that the same argument shows that if $f$ is an admissible monic,
then each $f^{P}$ is an admissible monic (i.e. injective) in $\mathsf{Vect}%
_{fd}(\widehat{F}_{P})$, and correspondingly if $f$ is an admissible epic,
then each $f^{P}$ is an admissible epic (i.e. surjective) in $\mathsf{Vect}%
_{fd}(\widehat{F}_{P})$. Since by definition the exact sequences of the
product exact category are just the levelwise exact ones, we obtain:

\begin{lemma}
The functor $\Xi$ of Equation \ref{lqa5} is exact and preserves duality.
\end{lemma}

Moreover, we claim:

\begin{lemma}
\label{Lemma_XiIsFaithful}The functor $\Xi$ is faithful.
\end{lemma}

\begin{proof}
Since the $\operatorname*{Hom}$-sets are abelian groups, it suffices to prove
that if some $f:V\rightarrow V^{\prime}$ has the property that $f^{P}=0$ for
all places $P$, then $f$ must have been zero. This result follows immediately
from the existence of $P$-coordinates:\footnote{The idea of $P$-coordinates is
easy to explain: We just exploit that a restricted product can always be
realized as a subset of the na\"{\i}ve plain product.\ The projection to the
individual factors are the $P$\emph{-coordinates}. While this is not a useful
technique to handle topological aspects, it is sufficient to test whether an
element is zero.} For $F=\mathbb{Q}$ this is proven in
\cite[(3.11)\ Definition]{MR637201} and our assertion is \cite[(3.12)\ Lemma]%
{MR637201}. We leave it to the reader to adapt the construction of such
coordinates to primes of $\mathcal{O}$ as opposed to the primes of
$\mathbb{Z}$.
\end{proof}

However, the functor $\Xi$ is not full. We refer the reader to Example
\ref{example_XiNotFull} for a concrete case.

\begin{definition}
We write $\mathsf{LCA}_{F,vfab}$ for the full subcategory of vector-free
adelic blocks inside $\mathsf{LCA}_{F,ab}$.
\end{definition}

\begin{lemma}
$\mathsf{LCA}_{F,vfab}$ is extension-closed in $\mathsf{LCA}_{F,ab}$ (and
$\mathsf{LCA}_{F}$). In particular, it is a fully exact subcategory of
$\mathsf{LCA}_{F,ab}$ (and $\mathsf{LCA}_{F}$).
\end{lemma}

\begin{proof}
If $G^{\prime}\hookrightarrow G\twoheadrightarrow G^{\prime\prime}$ is an
exact sequence with $G^{\prime},G^{\prime\prime}\in\mathsf{LCA}_{F,vfab}$,
then $G^{\prime},G^{\prime\prime}$ are topological torsion $\mathcal{O}%
$-modules by Lemma \ref{lemma_StructVectorFreeAdelicBlock}, and by
\cite[(3.17)]{MR637201} it follows that $G$ is topological torsion. Thus, by
Proposition \ref{prop_TopTorsionIsVectorFreeAdelicBlock} it follows that $G$
is also a vector-free adelic block. Now \cite[Lemma 10.20]{MR2606234}.
\end{proof}

Let $S$ be a finite set of places of $F$. Define a category%
\begin{equation}
\mathsf{LCA}_{F,vfab}^{(S)}:=\left(
\begin{array}
[c]{l}%
\text{\textsf{objects:} }(G,H)\mid G\in\mathsf{LCA}_{F,vfab}\text{,}\\
\qquad H\subseteq G\text{ a compact clopen }\mathcal{O}\text{-submodule}\\
\text{\textsf{morphisms:} }f:(G,H)\rightarrow(G^{\prime},H)\text{ in
}\mathsf{LCA}_{F,vfab}\text{ are morphisms}\\
\qquad f:G\rightarrow G^{\prime}\text{ in }\mathsf{LCA}_{F}\text{ such that we
have}\\
\qquad f^{P}(H^{P})\subseteq H^{\prime P}\text{ for all }P\notin S\text{.}%
\end{array}
\right) \label{def_LCAS}%
\end{equation}
As before, here $H^{P}$ (resp. $H^{\prime P}$) denotes the topological
$P$-torsion part for $P$ any place of $F$; and $f^{P}$ the induced morphism on
topological $P$-torsion parts (as in Equation \ref{lqa6}). There is a functor%
\[
\psi^{(S)}:\mathsf{LCA}_{F,vfab}^{(S)}\longrightarrow\mathsf{LCA}%
_{F,vfab}\text{,}\qquad\qquad(G,H)\longmapsto G\text{.}%
\]
We call a sequence in $\mathsf{LCA}_{F,vfab}^{(S)}$ exact if its image under
$\psi^{(S)}$ is exact. Analogously, we make it an exact category with duality
(in the sense of \cite[Definition 2.1]{MR2600285}) by sending $(G,H)$ to
$(G^{\vee},H^{\perp})$ with $H^{\perp}:=\ker(G^{\vee}\twoheadrightarrow
H^{\vee})$. Then $\psi^{(S)}$ is a duality preserving exact functor.

We observe that if $S\subseteq S^{\prime}$ is an inclusion of finite sets of
finite places of $F$, then there is a commutative diagram%
\begin{equation}%
\xymatrix{
{\mathsf{LCA}_{F,vfab}^{(S)}} \ar[rr] \ar[dr]_{\psi^{(S)}} & & {\mathsf
{LCA}_{F,vfab}^{(S^{\prime})}} \ar[dl]^{\psi^{(S^{\prime})}} \\
& {\mathsf{LCA}_{F,vfab}}.
}%
\label{lqa7}%
\end{equation}
As the exact structure on each $\mathsf{LCA}_{F,vfab}^{(S)}$ is just defined
so that it reflects exactness under the $\psi^{(-)}$, it is clear that all
functors in this diagram are exact. As the collection of finite subsets of
finite places is a filtering poset under inclusion, we can form the colimit of
exact categories%
\[
\mathsf{LCA}_{F,vfab}^{(\infty)}:=\underset{S}{\underrightarrow
{\operatorname*{colim}}}\,\mathsf{LCA}_{F,vfab}^{(S)}%
\]
and by the commutativity of Diagram \ref{lqa7} we obtain an exact functor%
\begin{equation}
\psi^{(\infty)}:\mathsf{LCA}_{F,vfab}^{(\infty)}\longrightarrow\mathsf{LCA}%
_{F,vfab}\label{lqa7b}%
\end{equation}
commuting with all $\psi^{(S)}$ along the natural functors in the inductive
diagram. This colimit exact category is explicitly given as follows:%
\begin{equation}
\mathsf{LCA}_{F,vfab}^{(\infty)}=\left(
\begin{array}
[c]{l}%
\text{\textsf{objects:} }(G,H)\mid G\in\mathsf{LCA}_{F,vfab}\text{,}\\
\qquad H\subseteq G\text{ a compact clopen }\mathcal{O}\text{-submodule}\\
\text{\textsf{morphisms:} }f:(G,H)\rightarrow(G^{\prime},H)\text{ in
}\mathsf{LCA}_{F,vfab}\text{ are morphisms}\\
\qquad f:G\rightarrow G^{\prime}\text{ in }\mathsf{LCA}_{F}\text{ such that
there exists a finite set}\\
\qquad S\text{ of finite places such that }f^{P}(H^{P})\subseteq H^{\prime
P}\text{ for all }P\notin S\text{.}%
\end{array}
\right) \label{lqa10}%
\end{equation}
However, this description can be simplified:

\begin{lemma}
\label{lemma_LatticeSys3}Let $(G,H)$ resp. $(G^{\prime},H^{\prime})$ be
objects in $\mathsf{LCA}_{F,vfab}$, and $H$ resp. $H^{\prime}$ compact clopen
$\mathcal{O}$-submodules. Then for \textsl{any} morphism $f:G\rightarrow
G^{\prime}$ in $\mathsf{LCA}_{F}$ there exists a finite set $S$ of finite
places of $F$ such that $f^{P}(H^{P})\subseteq H^{\prime P}$ for all $P\notin
S$.
\end{lemma}

\begin{proof}
Write $G^{\prime}$ as $\bigcup_{n\geq1}\frac{1}{n}C^{\prime}$ for $C^{\prime}$
a compact clopen, as exists by the definition of an adelic block. The image
$f(H)$ in $G^{\prime}$ is compact and thus by Lemma \ref{lemma_LatticeSys1}
lies inside some $\frac{1}{n}C^{\prime}$. We now have $f^{P}(H^{P}%
)\subseteq\left(  \frac{1}{n}C^{\prime}\right)  ^{P}$ by construction since a
topological $P$-torsion element can only be mapped to a topological
$P$-torsion element (if it had a non-zero image $x$ in the projection to any
other, say topological $P^{\prime}$-torsion, since $P$ and $P^{\prime}$ are
different, they are necessarily coprime, so $1\in P^{N}+P^{\prime N}$ for all
$N\geq1$, and then $1\cdot x$ lies in every open neighbourhood of zero in
$G^{\prime}$. Since $G^{\prime}$ is Hausdorff, this implies $x=0$ and gives a
contradiction). Finally, since $\frac{1}{n}C^{\prime}$ is a compact clopen
$\mathcal{O}$-submodule in $G^{\prime}$, just like $H^{\prime}$, we have
$\left(  \frac{1}{n}C^{\prime}\right)  ^{P}=H^{\prime P}$ for all but finitely
many places by Lemma \ref{lemma_LatticeSys2}. This proves our claim.
\end{proof}

Hence, we see that the last condition in the description of the morphisms of
$\mathsf{LCA}_{F,vfab}^{(\infty)}$ in Equation \ref{lqa10} is superfluous. We
record%
\[
\mathsf{LCA}_{F,vfab}^{(\infty)}=\left(
\begin{array}
[c]{l}%
\text{\textsf{objects:} }(G,H)\mid G\in\mathsf{LCA}_{F,vfab}\text{,}\\
\qquad H\subseteq G\text{ a compact clopen }\mathcal{O}\text{-submodule}\\
\text{\textsf{morphisms:} }f:(G,H)\rightarrow(G^{\prime},H^{\prime})\text{ in
}\mathsf{LCA}_{F,vfab}\text{ are morphisms}\\
\qquad f:G\rightarrow G^{\prime}\text{ in }\mathsf{LCA}_{F}\text{.}%
\end{array}
\right)  \text{.}%
\]

\begin{lemma}
\label{lemma_psiInftyIsEquivalence}The functor $\psi^{(\infty)}:\mathsf{LCA}%
_{F,vfab}^{(\infty)}\rightarrow\mathsf{LCA}_{F,vfab}$ is a duality preserving
exact equivalence of exact categories.
\end{lemma}

\begin{proof}
It is clear that the functor is exact. Next, we observe from the above
description that the morphisms do not depend on the extra data of
$H,H^{\prime}$ at all, so the functor is fully faithful. Finally, it is
essentially surjective:\ This just amounts to the statement that every
vector-free adelic block possesses a compact clopen $\mathcal{O}$-submodule at
all, and this is part of the definition of a vector-free adelic block.
\end{proof}

Next, we describe the categories $\mathsf{LCA}_{F,vfab}^{(S)}$ for $S$ finite.
For any ring $R$, let us write $P_{f}(R)$ for the exact category of finitely
generated projective $R$-modules. It is an exact category with duality by
letting $P^{\vee}:=\operatorname*{Hom}\nolimits_{R}(P,R)$.

\begin{proposition}
\label{prop_ProductCatLCAS}There is a duality preserving exact equivalence of
exact categories%
\[
\Xi^{(S)}:\mathsf{LCA}_{F,vfab}^{(S)}\longrightarrow\prod_{P\in S}%
\mathsf{Vect}_{fd}(\widehat{F}_{P})\times\prod_{P\notin S}P_{f}(\widehat
{\mathcal{O}}_{P})\text{.}%
\]

\end{proposition}

\begin{proof}
\textit{(Step 1)} First of all, we just define the functor at all. As the
notation suggests, we define this as a variation of $\Xi$: On objects we
define $\Xi^{(S)}$ as%
\[
(G,H)\longmapsto(G^{P})_{P\in S}\times(H^{P})_{P\notin S}\text{,}%
\]
i.e. for those $P$ in $S$, we take the entire topological $P$-torsion part. By
our variant of the Braconnier--Vilenkin Theorem (Theorem
\ref{thm_StrongBracVilenkin}), $G^{P}$ is a finite-dimensional $\widehat
{F}_{P}$-vector space, so this makes sense. Moreover, for all $P\notin S$ we
take $H^{P}:=G^{P}\cap H$, which is a full rank $\widehat{\mathcal{O}}_{P}%
$-lattice in $G^{P}$ by the same\ Theorem, so in particular $H^{P}$ is
finitely rank free $\widehat{\mathcal{O}}_{P}$-module, so it defines an object
in $P_{f}(\widehat{\mathcal{O}}_{P})$ as required. This defines $\Xi^{(S)}$ on
objects. Moreover, we send morphisms in $\mathsf{LCA}_{F,vfab}^{(S)}$ just to
their topological $P$-torsion parts $f^{P}$. On the $\mathsf{Vect}%
_{fd}(\widehat{F}_{P})$-factors it easily follows from continuity and the
$\widehat{F}_{P}$-vector space structure that $f^{P}$ is indeed $\widehat
{F}_{P}$-linear. On the $P_{f}(\widehat{\mathcal{O}}_{P})$ summands recall
that it was part of the definition of $\mathsf{LCA}_{F,vfab}^{(S)}$
(Definition \ref{def_LCAS}) that we have%
\[
f^{P}(H^{P})\subseteq H^{\prime P}\qquad\text{for all}\qquad P\notin S\text{,}%
\]
so $f^{P}$ is a well-defined map of $\widehat{\mathcal{O}}_{P}$%
-modules\ (again, first we get $\mathcal{O}$-linearity alone, and the
continuity and the $\widehat{\mathcal{O}}_{P}$-module structures allow us to
deduce that it must be an $\widehat{\mathcal{O}}_{P}$-module
homomorphism).\newline\textit{(Step 2)} Next, one checks that if $f$ is an
admissible monic (resp. epic), then all $f^{P}$ are injective (resp.
surjective). This suffices to know that the $f^{P}$ are levelwise admissible
monics (resp. epics) on the right hand side for all $P$, and thus admissible
monics (resp. epics) in the exact structure of the product exact category.
Thus, $\Xi^{(S)}$ is a well-defined exact functor.\newline\textit{(Step 3)} We
claim that $\Xi^{(S)}$ is also fully faithful. Faithfulness is easy to see; it
follows from the existence of $P$-coordinates and can be shown exactly as
Lemma \ref{Lemma_XiIsFaithful}. For fullness note that any morphism on the
right hand side can be written as a sum%
\[
f^{P}=(f^{P})^{S\text{-part}}+(f^{P})^{\text{not-}S\text{-part}}\text{,}%
\]
where $(f^{P})^{\text{not-}S\text{-part}}$ is zero on the factors with $P\in
S$, and reversely $(f^{P})^{S\text{-part}}$ is zero for the factors with
$P\notin S$. As $H,H^{\prime}$ are compact, we can prescribe any morphism on
the factors, so we can easily realize $(f^{P})^{\text{not-}S\text{-part}}$,
and $(f^{P})^{S\text{-part}}$ can be realized as $S$ is finite.\newline%
\textit{(Step 4)\ }Finally, we claim that $\Xi^{(S)}$ is essentially
surjective. To this end, given an object in the product exact category, say%
\[
G:=(G_{P})_{P\in S}\times(H_{P})_{P\notin S}\qquad\text{with}\qquad G_{P}%
\in\mathsf{Vect}_{fd}(\widehat{F}_{P})\text{,\quad}H_{P}\in P_{f}%
(\widehat{\mathcal{O}}_{P})\text{,}%
\]
define for $P\notin S$ the objects $G_{P}:=H_{P}\otimes_{\widehat{\mathcal{O}%
}_{P}}\widehat{F}_{P}$, and for the $P\in S$ choose a full rank $\widehat
{\mathcal{O}}_{P}$-lattice in $G_{P}$. Then take the restricted product%
\[
\tilde{G}:=\left.  \prod\nolimits^{\prime}\right.  (G_{P},H_{P})\text{,}%
\]
where $P$ runs over all finite places of $F$. This is easily checked to be an
object in $\mathsf{LCA}_{F,vfab}$ and we leave it to the reader to confirm
that $\Xi^{(S)}$ sends $\tilde{G}$ to an object isomorphic to $G$ in the
product exact category. Finally, as $\Xi^{(S)}$ is an essentially surjective,
fully faithful and exact functor between the two categories in our claim, we
conclude that $\Xi^{(S)}$ is an exact equivalence of exact categories.
\end{proof}

\begin{example}
\label{example_XiNotFull}Consider the vector-free part of the ad\`{e}les of
the rationals, $\left.  \prod\nolimits^{\prime}\right.  (\mathbb{Q}%
_{p},\mathbb{Z}_{p})$, where the restricted product runs over all prime
numbers $p $. In the product category $\prod_{p}\mathsf{Vect}_{fd}%
(\mathbb{Q}_{p})$ the image of this object has an endomorphism given by
multiplication by $p^{-1}$ in the $p$-th factor. If this was a morphism in
$\mathsf{LCA}_{\mathbb{Q}}$, it would have%
\[
\frac{1}{p}\mathbb{Z}_{p}\nsubseteq\mathbb{Z}_{p}%
\]
for all primes $p$, whereas by Lemma \ref{lemma_LatticeSys3} this property is
only possible for finitely many primes $p$.
\end{example}

\begin{theorem}
\label{thm_MainStructureThmOnKLCAFab}Let $F$ be a number field. Suppose $r$ is
the number of real places, and $s$ the number of complex places.

\begin{enumerate}
\item There is a canonically defined equivalence of $K$-theory spectra%
\[
K(\mathsf{LCA}_{F,ab})\overset{\sim}{\longrightarrow}\underset{S}%
{\underrightarrow{\operatorname*{colim}}}\prod_{P\in S}K(\widehat{F}%
_{P})\times\prod_{P\notin S}K(\widehat{\mathcal{O}}_{P})\times K(\mathbb{R}%
)^{r}\times K(\mathbb{C})^{s}\text{,}%
\]
where $S$ runs over the filtering poset of all finite sets of finite places of
$F$, ordered by inclusion.

\item For all integers $n$, there is a canonically defined isomorphism of
$K$-theory groups%
\[
K_{n}(\mathsf{LCA}_{F,ab})\cong\left\{  \left.  (\alpha_{P})_{P}\in\prod
_{P}K_{n}(\widehat{F}_{P})\right\vert
\begin{array}
[c]{l}%
\alpha_{P}\in K_{n}(\widehat{\mathcal{O}}_{P})\text{ for all but finitely}\\
\text{many among the finite places}%
\end{array}
\right\}  \text{.}%
\]
Here the product is taken over all places $P$, both finite and infinite, and
the condition makes sense since $K_{n}(\widehat{\mathcal{O}}_{P})\subseteq
K_{n}(\widehat{F}_{P})$ is naturally a subgroup.
\end{enumerate}
\end{theorem}

\begin{proof}
\textit{(Step 1)\ }We leave it to the reader to show that%
\[
K(\mathsf{LCA}_{F,ab})\overset{\sim}{\longrightarrow}K(\mathsf{LCA}%
_{F,vfab})\times K(\mathbb{R})^{r}\times K(\mathbb{C})^{s}\text{,}%
\]
which reduces the proof to handling the vector-free adelic blocks.\newline%
\textit{(Step 2) }By Lemma \ref{lemma_psiInftyIsEquivalence} we have an
equivalence in $K$-theory%
\begin{equation}
\psi^{(\infty)}:K(\mathsf{LCA}_{F,vfab}^{(\infty)})\overset{\sim
}{\longrightarrow}K(\mathsf{LCA}_{F,vfab})\label{lqa11}%
\end{equation}
since this morphism is induced from an exact equivalence of exact categories.
However, by the definition of the left hand side we had%
\[
\mathsf{LCA}_{F,vfab}^{(\infty)}:=\underset{S}{\underrightarrow
{\operatorname*{colim}}}\,\mathsf{LCA}_{F,vfab}^{(S)}%
\]
and since $K$-theory commutes with filtering colimits of exact categories
\cite[Chapter IV, Elem. Properties 6.4]{MR3076731}, this promotes the
equivalence of Equation \ref{lqa11} to%
\begin{equation}
\underset{S}{\underrightarrow{\operatorname*{colim}}}\,K(\mathsf{LCA}%
_{F,vfab}^{(S)})\overset{\sim}{\longrightarrow}K(\mathsf{LCA}_{F,vfab}%
)\text{.}\label{lqa12}%
\end{equation}
However, by Proposition \ref{prop_ProductCatLCAS} we have both on the level of
spectra,%
\[
K(\mathsf{LCA}_{F,vfab}^{(S)})\overset{\sim}{\longrightarrow}\prod_{P\in
S}K(\widehat{F}_{P})\times\prod_{P\notin S}K(\widehat{\mathcal{O}}_{P})
\]
and on the level of homotopy groups,%
\[
K_{n}(\mathsf{LCA}_{F,vfab}^{(S)})\cong\prod_{P\in S}K_{n}(\widehat{F}%
_{P})\times\prod_{P\notin S}K_{n}(\widehat{\mathcal{O}}_{P})
\]
since $K$-theory commutes with (even infinite!) products of exact categories,
see Carlsson's Theorem \cite[Theorem 3.13]{MR1351941} (or \cite{kaspwinges}
for a different proof). The latter is not at all obvious, and rather
surprising. Now, consider the colimit in Equation \ref{lqa12}. It is easy to
see that for an inclusion $S\subseteq S^{\prime}$ of finite sets of finite
places, the induced morphism%
\begin{equation}
K_{n}(\mathsf{LCA}_{F,vfab}^{(S)})\longrightarrow K_{n}(\mathsf{LCA}%
_{F,vfab}^{(S^{\prime})})\label{lqa21}%
\end{equation}
induced from the exact functor given by the horizontal arrow in Diagram
\ref{lqa7} is factor-wise induced from the functor%
\[
P_{f}(\widehat{\mathcal{O}}_{P})\longrightarrow\mathsf{Vect}_{fd}(\widehat
{F}_{P})\text{,}\qquad M\longmapsto M\otimes_{\widehat{\mathcal{O}}_{P}%
}\widehat{F}_{P}\text{,}%
\]
and thus from the standard map $K_{n}(\widehat{\mathcal{O}}_{P})\rightarrow
K_{n}(\widehat{F}_{P})$, induced by the inclusion into the field of fractions.
However, since the residue field of $\widehat{\mathcal{O}}_{P}$ is finite,
this morphism is known to be injective, see \cite[Chapter V, Corollary
6.9.2]{MR3076731} (this is a very special case of the Gersten Conjecture). It
follows that all the transition morphisms as in Equation \ref{lqa21} are injective.
\end{proof}

\begin{example}
[$K(\mathbb{A})\neq K(\mathsf{LCA}_{F,ab})$]%
\label{example_KAdifferentFromKAdelicBlocks}There is a natural functor%
\[
G:P_{f}(\mathbb{A})\rightarrow\mathsf{LCA}_{F,ab}%
\]
sending the projective generator $\mathbb{A}$ to itself, but now regarded as a
locally compact $F$-module. For many rings $R$, one has $K_{1}(R)\cong
R^{\times}$, so in view of $K_{1}(\mathsf{LCA}_{F,ab})\cong\mathbb{A}^{\times
}$ it is tempting to hope that the functor $G$ might induce an equivalence on
the level of $K$-theory. This hope is plausible both if we consider arbitrary
algebraic $\mathbb{A}$-modules on the left, or topological $\mathbb{A}%
$-modules. However, it is false in either case. To see this, recall that
$P_{f}(\mathbb{A})$ is the idempotent completion of the exact category of free
finitely generated $\mathbb{A}$-modules. Thus, $G$ is the idempotent
completion of the functor%
\[
P_{f}^{\operatorname*{free}}(\mathbb{A})\rightarrow\mathsf{LCA}_{F,ab}\text{.}%
\]
On the level of $K_{0}$-groups, this functor induces the map going to%
\begin{equation}
K_{0}(\mathsf{LCA}_{F,ab})\cong\left\{  (\alpha_{P})_{P}\in\prod_{P}%
\mathbb{Z}\right\}  \text{,}\label{lint_252}%
\end{equation}
using the identification of Theorem \ref{thm_MainStructureThmOnKLCAFab} and
the fact that $K_{0}(\widehat{\mathcal{O}}_{P})\cong K(\widehat{F}_{P}%
)\cong\mathbb{Z}$. The free rank one module $\mathbb{A}$ gets sent to the
constant vector $(1,1,\ldots)$. As such, irrespective what object $X\in
P_{f}^{\operatorname*{free}}(\mathbb{A})$ we consider, the output is a vector
which admits a uniform upper bound $B_{X}$ such that $\left\vert \alpha
_{P}\right\vert \leq B_{X}$ for all places $B$. Now, the idempotent completion
$P_{f}(\mathbb{A})$ may have a much larger image, but nonetheless whenever
$X\in P_{f}(\mathbb{A})$ is a direct summand of a free finite rank
$\mathbb{A}$-module, the image will be a corresponding direct summand. Hence,
since the image of $K_{0}(P_{f}^{\operatorname*{free}}(\mathbb{A}))$ only
consists of vectors with entries uniformly bounded among all places, this
property is still true for all direct summands. Thus, no vector in the right
hand side of Equation \ref{lint_252} with unbounded entries can lie in the image.
\end{example}

\begin{remark}
The problem raised in Example \ref{example_KAdifferentFromKAdelicBlocks} is
similar to the fact that $K$-theory does not commute with infinite products of
rings. For example, let $R$ be some ring. While $K$-theory commutes with
infinite products of exact categories by \cite{MR1351941}, the category of
finitely generated projective modules over the product ring $\prod_{\aleph
_{0}}R$ is much smaller than the product exact category $\prod_{\aleph_{0}%
}P_{f}(R)$. The reason is analogous to the one underlying Example
\ref{example_KAdifferentFromKAdelicBlocks}: In the latter category there is no
uniform bound of ranks, while there is such a bound for the former.
\end{remark}

\begin{example}
\label{example8}Instead of restricting to finitely generated projective
modules as in Example \ref{example_KAdifferentFromKAdelicBlocks}, one might
also dream about a functor $\mathsf{Mod}_{\mathbb{A}}\rightarrow
\mathsf{LCA}_{F,ab}$, sending finitely generated (or finitely presented, or
coherent,\ldots) $\mathbb{A}$-modules to themselves regarded as locally
compact $F$-modules. While such an association might work on the level of
objects, this will not give an exact functor. To see this, consider
$F:=\mathbb{Q}$ and the category $\mathsf{Mod}_{\mathbb{A}}$ of algebraic
$\mathbb{A}$-modules. As $\mathbb{A}$ is a free module, it is clearly
projective. Next, define an ad\`{e}le%
\[
r:=(2,3,5,\ldots,p,\ldots,1)\text{,}%
\]
i.e. with the prime $p$ at $\mathbb{Q}_{p}$, and $1$ at the real place
$\mathbb{R}$. This is an integral ad\`{e}le. Note that it is not an id\`{e}le,
so it is not a unit in the ring. We get a short exact sequence%
\begin{equation}
\mathbb{A}\overset{\cdot r}{\hookrightarrow}\mathbb{A}\twoheadrightarrow
Q\qquad\text{in}\qquad\mathsf{Mod}_{\mathbb{A}}\label{lint_535}%
\end{equation}
and in fact $Q=\left(  \prod\limits_{p}\mathbb{F}_{p}\right)  \otimes
_{\mathbb{Z}}\mathbb{Q}$, which is an infinite-dimensional $\mathbb{Q}$-vector
space. Now, our conjectural functor would send this sequence to itself, so to
be exact, we would need $r\mathbb{A}\subset\mathbb{A}$ to be a closed
injection. However, this is not the case: The complement has the description%
\begin{align*}
U:=\mathbb{A}\setminus r\mathbb{A}  & =\left\{  (\alpha_{p})_{p}\in
\mathbb{A}\mid\text{there are infinitely many }p\text{ such that }\alpha
_{p}\notin p\mathbb{Z}_{p}\right\} \\
& =\left\{  (\alpha_{p})_{p}\in\mathbb{A}\mid\text{there are infinitely many
}p\text{ such that }\alpha_{p}\in\mathbb{Z}_{p}^{\times}\right\}  \text{.}%
\end{align*}
Clearly $1_{\mathbb{A}}=(1,1,1\ldots)\in U$. If $U$ is open, some open
neighbourhood of $1_{\mathbb{A}}$ is contained in $U$. If $\mathbb{O}$ denotes
the integral ad\`{e}les, the sets $1_{\mathbb{A}}+\frac{1}{n}\mathbb{O}$ form
an open neighbourhood basis. However, any choice of $n\in\mathbb{Z} $ has only
finitely many prime factors, so for any choice%
\[
1_{\mathbb{A}}+\frac{1}{n}\mathbb{O}=(1+\frac{1}{n}\mathbb{Z}_{2}%
,\ldots,1+\frac{1}{n}\mathbb{Z}_{p},\ldots,1+\mathbb{Z}_{p^{\prime}}%
,\ldots\mathbb{)}\text{,}%
\]
where $p^{\prime}$ runs through all primes strictly larger than any divisor of
$n$. However, $1+\mathbb{Z}_{p}=\mathbb{Z}_{p}$ and so there are at most
finitely many components whose elements $\alpha_{p}$ lie in the units
$\mathbb{Z}_{p}^{\times}$. As a result, none of these sets is contained in $U
$.
\end{example}

\begin{example}
\label{example_Proj}There is a further philosophical point which stresses the
difference between the exact categories $\mathsf{Mod}_{\mathbb{A}}$ and
$\mathsf{LCA}_{F,ab}$, at least when $\mathbb{A}$ is regarded merely as an
abstract ring without topology. In the former category, the ad\`{e}les
$\mathbb{A}$ are a projective object. However, they are not an injective
module: If $\mathbb{A}$ were an injective module, the injection of Equation
\ref{lint_535} would have to split, which would force all ad\`{e}les to be
divisible by $r$. However, clearly there is no ad\`{e}le $\alpha$ with
$r\cdot\alpha=(1,1,1\ldots)$. Thus, there is a projective--injective asymmetry
in the category $\mathsf{Mod}_{\mathbb{A}}$. On the other hand, in
$\mathsf{LCA}_{F}$ the ad\`{e}les are self-dual (one of the cornerstones of
Tate's thesis). Thus, they are either both injective and projective, or neither.
\end{example}

We could not settle the following:

\begin{conjecture}
Every adelic block $G\in\mathsf{LCA}_{F}$ is both an injective and projective object.
\end{conjecture}

As explained in Example \ref{example_Proj}, it suffices to prove that they are
all injective (or all projective) since the dual of an adelic block is again
an adelic block\ by Corollary \ref{Cor_AdelicBlocksAreClosedUnderDuality}, but
adelic blocks are also isomorphic to their duals by Theorem
\ref{thm_StrongBracVilenkin}.

\begin{dream}
We conclude this section with a vague dream. Suppose $X/\mathbb{F}_{q}$ is a
smooth variety. One could hope that there are exact categories of
\textquotedblleft higher adelic blocks\textquotedblright\ attached to $X$ such
that the restricted products appearing in \cite{MR3291352} for the cycle
module $M:=K(-)$ being chosen to be $K$-theory, for example in Equation 1.7.
\textit{loc. cit.}, or \S 4 \textit{loc. cit.}, naturally would appear as the
$K$-theory of these higher adelic blocks. Theorem
\ref{thm_MainStructureThmOnKLCAFab} then could be regarded as the
one-dimensional case of an arithmetic variant. The finiteness conditions
\textit{loc. cit.} were made rather ad hoc, so perhaps some modifications will
be necessary, possibly resulting in a slightly different flasque resolution.
\end{dream}

\section{Main theorems}

We write $\mathsf{Vect}_{F}$ for the abelian category of all $F$-vector spaces
(without any topology), and $\mathsf{Vect}_{F,fd}$ for the abelian full
subcategory of all finite-dimensional $F$-vector spaces.

We have already shown that $\mathsf{LCA}_{F,qab+fd}$ is left $s$-filtering in
$\mathsf{LCA}_{F}$ (Theorem \ref{thm_qabfdIsLeftSpecial}), so the quotient
exact category in the following claim exists.

\begin{lemma}
\label{lemma_fmodfdisallmodqabfd}There is an exact equivalence of exact
categories%
\[
\Phi:\mathsf{Vect}_{F}/\mathsf{Vect}_{F,fd}\overset{\sim}{\longrightarrow
}\mathsf{LCA}_{F}/\mathsf{LCA}_{F,qab+fd}\text{.}%
\]

\end{lemma}

\begin{proof}
We define the functor $\Phi$ by sending a vector space $V$ to itself, equipped
with the discrete topology. This functor is fully faithful: Given
$V,V^{\prime}\in\mathsf{Vect}_{F}$, the map%
\begin{equation}
\operatorname*{Hom}(V,V^{\prime})\longrightarrow\operatorname*{Hom}%
(\Phi(V),\Phi(V^{\prime}))\label{laa7}%
\end{equation}
goes to the localization $\mathsf{LCA}_{F}/\mathsf{LCA}_{F,qab+fd}%
:=\mathsf{LCA}_{F}[\Sigma_{e}^{-1}]$ on the right, where $\Sigma_{e}$ is the
collection of admissible epics with kernel in $\mathsf{LCA}_{F,qab+fd}$ (see
\cite[Prop. 2.19]{MR3510209}). Since $\Phi(V)$, $\Phi(V^{\prime})$ are
discrete, it suffices to consider $K\hookrightarrow D\twoheadrightarrow
D^{\prime}$ with $D,D^{\prime}$ discrete and $K\in\mathsf{LCA}_{F,qab+fd}$.
However, as this sequence is exact, $K$ must also be discrete. However, an
object which is both discrete and in $\mathsf{LCA}_{F,qab+fd}$ must be a
discrete finite-dimensional $F$-vector space (apply Lemma
\ref{lemma_nomapsfromqabtodiscrete} to the identity map), so we localize by
the system of arrows $D\twoheadrightarrow D^{\prime}$ with finite-dimensional
kernel. Of course, for forming the quotient $\mathsf{Vect}_{F}/\mathsf{Vect}%
_{F,fd}$, we also localize at the morphisms with finite-dimensional kernel, so
the Hom-sets are the same. It follows that the map in Equation \ref{laa7} is
an isomorphism of abelian groups, and thus $\Phi$ is fully faithful. Moreover,
$\Phi$ is essentially surjective:\ Given any object $G\in\mathsf{LCA}_{F}$, by
Proposition \ref{thm_StructLCAF} there is an exact sequence $Q\hookrightarrow
G\overset{f}{\twoheadrightarrow}D$ with $Q\in\mathsf{LCA}_{F,qab}$ and $D$
discrete. However, this is an admissible epic with quasi-adelic kernel, so
$f\in\Sigma_{e}$, i.e. $G\cong D$ in the quotient $\mathsf{LCA}_{F}%
/\mathsf{LCA}_{F,qab+fd}$, but $D$ lies in the image of the functor $\Phi$.
\end{proof}

Moreover, we need the following similar computation. Again, we have already
shown that $\mathsf{LCA}_{F,C}$ is left $s$-filtering in $\mathsf{LCA}%
_{F,qab}$ (Proposition \ref{prop_LCAFCLeftFilteringInLCAFQab}).

\begin{lemma}
\label{lemma_Fad_equivalent_to_FqabmodFC}The natural functor%
\[
\mathsf{LCA}_{F,ab}\overset{\sim}{\longrightarrow}\mathsf{LCA}_{F,qab}%
/\mathsf{LCA}_{F,C}\text{,}%
\]
sending an adelic block to itself, induces an exact equivalence of exact categories.
\end{lemma}

\begin{proof}
(1)\ Exactness is clear since a sequence in $\mathsf{LCA}_{F,qab}%
/\mathsf{LCA}_{F,C}$ is exact if it comes from an exact sequence in
$\mathsf{LCA}_{F,qab} $, and this is clearly true coming from $\mathsf{LCA}%
_{F,ab}$, (2) essential surjectivity is clear since by Proposition
\ref{prop_DecomposeQuasiAdelicBlock} each quasi-adelic block has the shape
$V\oplus C$ with $C$ compact and $V$ an adelic block, so in the quotient
category on the right each object is isomorphic to an adelic block, (3) full
faithfulness follows from the fact $\mathsf{LCA}_{F,qab}/\mathsf{LCA}%
_{F,C}=\mathsf{LCA}_{F,qab}[\Sigma_{e}^{-1}]$, where $\Sigma_{e}$ is the
collection of admissible epics with kernel in $\mathsf{LCA}_{F,C}$, that is on
the strict image of the functor it suffices to consider sequences%
\[
C\hookrightarrow V\twoheadrightarrow V^{\prime}%
\]
with $V,V^{\prime}$ adelic blocks and $C$ a compact $F$-space. However,
presenting $V=\bigcup_{n\geq1}\frac{1}{n}\tilde{C}$ with $\tilde{C}$ a clopen
$\mathcal{O}$-submodule, the latter defines an open cover of $V$. The image of
$C$ in $V$ is compact, so there exists some $n$ such that $C\subseteq\frac
{1}{n}\tilde{C}$ by compactness, and then $\bigcap_{m\geq1}\frac{1}%
{m}C\subseteq\bigcap_{m\geq1}\frac{1}{nm}\tilde{C}=0$, where we used that $V$
is an adelic block. However, as $C$ is an $F$-vector space, $\frac{1}{m}C=C$
and thus the left hand side is just $C$ itself. The argument thus shows that
$C=0$. This shows that on the strict image we are only inverting isomorphisms,
proving full faithfulness. (There is also an alternative argument: For
$C\hookrightarrow V$ consider its dual morphism $V^{\vee}\twoheadrightarrow
C^{\vee}$. As the dual $V^{\vee}$ is still an adelic block by Corollary
\ref{Cor_AdelicBlocksAreClosedUnderDuality}, but $C^{\vee}$ is discrete, Lemma
\ref{lemma_nomapsfromqabtodiscrete} forces this map to be zero. Being an epic,
this forces $C^{\vee}=0$ and thus $C=0$.)
\end{proof}

Next, we shall employ the defining properties of localizing invariants, as
induced from the following crucial fact due to Schlichting.

\begin{theorem}
[Schlichting Localization]\label{thm_SchlichtingLocalizationThm}Suppose
$\mathsf{C}$ is an idempotent complete exact category and $\mathsf{C}%
\hookrightarrow\mathsf{D}$ a left $s$-filtering full subcategory of an exact
category $\mathsf{D}$. Then%
\[
D^{b}(\mathsf{C})\hookrightarrow D^{b}(\mathsf{D})\twoheadrightarrow
D^{b}(\mathsf{D}/\mathsf{C})
\]
is an exact sequence of triangulated categories.
\end{theorem}

We refer to \cite[Prop. 2.6]{MR2079996} for the proof.

\begin{theorem}
\label{thm_MainKOfLCAF}Let $F$ be a number field and let
$K:\operatorname*{Cat}_{\infty}^{\operatorname*{ex}}\rightarrow\mathsf{A}$ be
a localizing invariant with values in a stable $\infty$-category $\mathsf{A}$
(in the sense of \cite{MR3070515}). Then there is a canonical fiber sequence%
\[
K(F)\longrightarrow K(\mathsf{LCA}_{F,ab})\longrightarrow K(\mathsf{LCA}%
_{F})\text{,}%
\]
and the first arrow is induced from the exact functor sending $F$ to the
ad\`{e}les $\mathbb{A}$.
\end{theorem}

\begin{proof}
The proof is modelled after the analogous computation in \cite{obloc}. Only
the ingredient categories as well as some computational details
change.\newline\textit{(Step 1)} By Theorem \ref{thm_qabfdIsLeftSpecial} the
full subcategory $\mathsf{LCA}_{F,qab+fd}$ is left $s$-filtering in
$\mathsf{LCA}_{F}$. Thus, by the work of Schlichting \cite{MR2079996} there
exists a quotient exact category $\mathsf{LCA}_{F}/\mathsf{LCA}_{F,qab+fd}$
and we have an exact sequence of exact categories%
\[
\mathsf{LCA}_{F,qab+fd}\hookrightarrow\mathsf{LCA}_{F}\twoheadrightarrow
\mathsf{LCA}_{F}/\mathsf{LCA}_{F,qab+fd}\text{.}%
\]
We get a commutative diagram of exact categories with exact functors,%
\[%
\xymatrix{
{\mathsf{Vect}_{F,fd}} \ar@{^{(}->}[r] \ar[d] & \mathsf{Vect}_{F} \ar@
{->>}[r] \ar[d] & {\mathsf{Vect}_{F}}/{\mathsf{Vect}_{F,fd}} \ar[d]_{\sim} \\
\mathsf{LCA}_{F,qab+fd} \ar@{^{(}->}[r] & \mathsf{LCA}_{F} \ar@{->>}%
[r] & {\mathsf{LCA}_{F}/{\mathsf{LCA}_{F,qab+fd}},
}}%
\]
where the downward arrows just send an untopologized $F$-vector space to
itself, equipped with the discrete topology. This is an exact functor.
Moreover, the functor on the right is an exact equivalence by Lemma
\ref{lemma_fmodfdisallmodqabfd}. Using Schlichting's Localization Theorem
(Theorem \ref{thm_SchlichtingLocalizationThm}), the rows induce fiber
sequences in the localizing theory $K$, and the downward exact functors induce
morphisms of the output of $K$. We get a commutative diagram in $\mathsf{A}$,%
\[%
\xymatrix{
K({\mathsf{Vect}_{F,fd}}) \ar[r] \ar[d] & K(\mathsf{Vect}_{F}) \ar
[r] \ar[d] & K({\mathsf{Vect}_{F}}/{\mathsf{Vect}_{F,fd}}) \ar[d]_{\sim} \\
K(\mathsf{LCA}_{F,qab+fd}) \ar[r] & K(\mathsf{LCA}_{F}) \ar[r] & K({\mathsf
{LCA}_{F}/{\mathsf{LCA}_{F,qab+fd}}),
}}%
\]
where the right downward arrow now has become an equivalence. It follows that
the left square is bi-Cartesian in $\mathsf{A}$. Furthermore, by the Eilenberg
swindle we have $K(\mathsf{Vect}_{F})=0$. Thus, we obtain a fiber sequence%
\begin{equation}
K(F)\longrightarrow K(\mathsf{LCA}_{F,qab+fd})\longrightarrow K(\mathsf{LCA}%
_{F})\text{.}\label{l_bstep1}%
\end{equation}
\textit{(Step 2) }Next, we shall show that every object in $\mathsf{LCA}%
_{F,qab+fd}$ has a two-term resolution by objects in the full subcategory
$\mathsf{LCA}_{qab}$. To this end, note that by definition any object can be
written as a direct sum $Q\oplus D$ with $Q\in\mathsf{LCA}_{F,qab}$ and $D$
finite-dimensional discrete. In other words, $D$ is a finite direct sum of
copies of the number field $F$, equipped with the discrete topology. We have
the ad\`{e}le sequence%
\[
F\hookrightarrow\mathbb{A}\twoheadrightarrow\mathbb{A}/F\text{,}%
\]
where $\mathbb{A}$, $\mathbb{A}/F\in\mathsf{LCA}_{F,qab}$. Taking direct sums,
this proves Condition C1 of \cite[Theorem 12.1]{MR1421815}. By the Remark
following the Theorem \textit{loc. cit.}, in order to check Condition C2
\textit{loc. cit.}, it suffices to check that $\mathsf{LCA}_{F,qab}$ is closed
under cokernels of admissible monics in $\mathsf{LCA}_{F}$. We have already
shown this in Lemma \ref{lemma_Fqab_closed_under_cokernels}. Applying the
Theorem, we get that the exact functor%
\[
\mathsf{LCA}_{F,qab}\longrightarrow\mathsf{LCA}_{F,qab+fd}%
\]
induces a derived equivalence, so $K(\mathsf{LCA}_{F,qab})\overset{\sim
}{\longrightarrow}K(\mathsf{LCA}_{F,qab+fd})$ is an equivalence in $\mathsf{A}
$.\newline\textit{(Step 3) }We had shown in Proposition
\ref{prop_LCAFCLeftFilteringInLCAFQab} that $\mathsf{LCA}_{F,C}$ is left
$s$-filtering in $\mathsf{LCA}_{F,qab}$, so again the quotient exact category
exists and we get the exact sequence%
\[
\mathsf{LCA}_{F,C}\hookrightarrow\mathsf{LCA}_{F,qab}\twoheadrightarrow
\mathsf{LCA}_{F,qab}/\mathsf{LCA}_{F,C}%
\]
of exact categories. Using Schlichting's Localization Theorem once more, and
$K(\mathsf{LCA}_{F,C})=0$ by the Eilenberg swindle, as $\mathsf{LCA}_{F,C}$ is
closed under infinite products by Tychonoff's Theorem, we deduce the first
equivalence in%
\[
K(\mathsf{LCA}_{F,qab})\overset{\sim}{\longrightarrow}K(\mathsf{LCA}%
_{F,qab}/\mathsf{LCA}_{F,C})\overset{\sim}{\longrightarrow}K(\mathsf{LCA}%
_{F,ab})\text{,}%
\]
and the second equivalence comes from Lemma
\ref{lemma_Fad_equivalent_to_FqabmodFC}. Finally, take the fiber sequence in
Equation \ref{l_bstep1} and combining it with the equivalences from Step 2 and
Step 3, it transforms into%
\[
K(F)\longrightarrow K(\mathsf{LCA}_{F,ab})\longrightarrow K(\mathsf{LCA}%
_{F})\text{.}%
\]
\newline\textit{(Step 4)} In order to prove our claim, it remains to identify
that the arrow%
\[
K(F)\longrightarrow K(\mathsf{LCA}_{F,ab})
\]
is indeed induced from sending the projective generator $F$ of $\mathsf{Vect}%
_{F,fd}$ to the ad\`{e}les $\mathbb{A}$ in $\mathsf{LCA}_{F,ab}$. To this end,
we need to trace through our constructions. In Step 1 the arrow
$K(F)\rightarrow K(\mathsf{LCA}_{F,qab+fd})$ comes from sending $F$ to itself,
equipped with the discrete topology, call this map $p_{1}$. Next, as our
localizing invariant $K$ is also additive (see \cite[Section 8]{MR3070515}),
we know that the functor%
\[
\mathsf{Vect}_{F,fd}\longrightarrow\mathcal{E}\mathsf{LCA}_{F,qab+fd}\text{,}%
\]
where $\mathcal{E}\mathsf{LCA}_{F,qab+fd}$ denotes the exact category of short
exact sequence in $\mathsf{LCA}_{F,qab+fd}$,%
\[
(p_{1},p_{2},p_{3}):F\longmapsto\left[  F\hookrightarrow\mathbb{A}%
\twoheadrightarrow\mathbb{A}/F\right]
\]
has the property that the induced morphisms%
\[
p_{i\ast}:K(F)\longrightarrow K(\mathsf{LCA}_{F,qab+fd})\qquad\text{(for
}i=1,2,3\text{)}%
\]
in $\mathsf{A}$ satisfy $p_{1\ast}+p_{3\ast}=p_{2\ast}$. However, since
$\mathbb{A}/F\simeq F^{\vee}$ is compact, the functor $p_{3}$ can be factored
as%
\[
p_{3}:\mathsf{Vect}_{F,fd}\longrightarrow\mathsf{LCA}_{F,C}\longrightarrow
\mathsf{LCA}_{F,qab+fd}%
\]
and since $K(\mathsf{LCA}_{F,C})=0$ by the Eilenberg swindle, we have
$p_{3\ast}=0$. It follows that $p_{1\ast}=p_{2\ast}$, but $p_{1\ast}$ is the
map induced from Step 1, and $p_{2\ast}$ is the functor sending $F$ to
$\mathbb{A}$, so the functor of our claim. This proves the theorem.
\end{proof}

\section{Agreement theorems}

We return to the discussion of the introduction and Clausen's theory
\cite{clausen}. For his generalization of the Artin reciprocity map, Clausen
constructs canonical maps $\psi$ as in
\[%
\begin{array}
[c]{rcl}%
\mathbb{Z}\overset{\psi}{\longrightarrow}\pi_{1}K(\mathsf{LCA}_{F}%
)\longrightarrow\operatorname*{Gal}(F^{\operatorname*{ab}}/F) & \qquad &
\text{if }F\text{ is a finite field,}\\
F^{\times}\overset{\psi}{\longrightarrow}\pi_{1}K(\mathsf{LCA}_{F}%
)\longrightarrow\operatorname*{Gal}(F^{\operatorname*{ab}}/F) & \qquad &
\text{if }F\text{ is a local field,}\\
\mathbb{A}^{\times}/F^{\times}\overset{\psi}{\longrightarrow}\pi
_{1}K(\mathsf{LCA}_{F})\longrightarrow\operatorname*{Gal}%
(F^{\operatorname*{ab}}/F) & \qquad & \text{if }F\text{ is a number field,}%
\end{array}
\]
and shows that the compositions agree with the usual reciprocity map. See
\cite[Corollary 3.29]{clausen} for the construction of $\psi$ via a universal
mapping property. The global reciprocity map is glued from the local ones, as
usual, so proving agreement reduces to showing local agreement as well as
vanishing on the principal id\`{e}les. This is no different from the classical
local-to-global approach. We prove statements of the shape as predicted in
Remark 3.30 \textit{loc. cit.}, with a slight correction in the case of number
fields (see Example \ref{example_KAdifferentFromKAdelicBlocks}), and a
topologized version in the case of $p$-adic local fields:

\begin{theorem}
[Agreement with classical class field theory]\label{thm_Main_Agreement}Let $K$
denote non-connective $K$-theory.

\begin{enumerate}
\item If $F$ is a finite field, then%
\[
K_{1}(\mathsf{LCA}_{F})\cong\mathbb{Z}\text{.}%
\]

\item If $F$ is a $p$-adic local field, then%
\[
K_{1}(\mathsf{LCA}_{F,\operatorname*{top}})\cong F^{\times}\text{.}%
\]

\item If $F$ is a number field,%
\[
K_{1}(\mathsf{LCA}_{F})\cong\mathbb{A}^{\times}/\iota(F^{\times})\text{,}%
\]
i.e. there is a canonical isomorphism to Chevalley's id\`{e}le class group.
\end{enumerate}
\end{theorem}

\begin{proof}
\textit{(Claim 1)} We just use Proposition \ref{prop_KLCA_for_finite_fields}.
In particular, $K_{1}(\mathsf{LCA}_{\mathbb{F}_{q}})=K_{0}(\mathbb{F}%
_{q})\cong\mathbb{Z}$.\newline\textit{(Claim 2)} By Proposition
\ref{prop_LCA_F_overlocalfield} there is an equivalence in non-connective
$K$-theory%
\[
K(\mathsf{LCA}_{F,\operatorname*{top}})\overset{\sim}{\longrightarrow
}K(F)\text{,}%
\]
where $F$ is the local field. Thus, $K_{1}(\mathsf{LCA}_{F,\operatorname*{top}%
})\cong K_{1}(F)\cong F^{\times}$, as claimed.\newline\textit{(Claim 3)}
Theorem \ref{thm_MainKOfLCAF}, applied where $K$ refers to non-connective
$K$-theory, gives us a fiber sequence in spectra. The long exact sequence in
homotopy groups of this sequence yields the excerpt%
\begin{equation}
K_{1}(F)\overset{\alpha_{1}}{\longrightarrow}K_{1}(\mathsf{LCA}_{F,ab}%
)\longrightarrow K_{1}(\mathsf{LCA}_{F})\longrightarrow K_{0}(F)\overset
{\alpha_{0}}{\longrightarrow}K_{0}(\mathsf{LCA}_{F,ab})\label{lqa16}%
\end{equation}
and we know that $\alpha_{i}$ is induced from the functor%
\[
\mathsf{Vect}_{fd}(F)\longrightarrow\mathsf{LCA}_{F,ab}\text{,}\qquad\qquad
F\longmapsto\mathbb{A}\text{,}%
\]
where the functor is described by its function on a projective generator of
the category on the left. Using the isomorphism of Theorem
\ref{thm_MainStructureThmOnKLCAFab} on the level of $K_{0}$, we obtain that
$K_{0}(F)\rightarrow K_{0}(\mathsf{LCA}_{F,ab})$ unravels as%
\[
\mathbb{Z}\longrightarrow\prod\nolimits_{P}\mathbb{Z}\text{,}\qquad
\qquad1\longmapsto(1,1,\ldots)
\]
since $K_{0}(\widehat{\mathcal{O}}_{P})=K_{0}(\widehat{F}_{P})\cong\mathbb{Z}%
$, and we deduce that $\alpha_{0}$ is injective. Similarly, for $\alpha_{1}$
we obtain that $K_{1}(F)\rightarrow K_{1}(\mathsf{LCA}_{F,ab})$ unravels as%
\[
F^{\times}\longrightarrow\left\{  \left.  (\alpha_{P})_{P}\in\prod_{P}%
\widehat{F}_{P}^{\times}\right\vert \text{ }%
\begin{array}
[c]{l}%
\alpha_{P}\in\widehat{\mathcal{O}}_{P}^{\times}\text{ for all but finitely}\\
\text{many finite places}%
\end{array}
\right\}  \cong\mathbb{A}^{\times}%
\]
and the map is the diagonal, call it $\iota$. Again, this is clearly
injective. Thus, the exact sequence in Equation \ref{lqa16} implies the claim.
\end{proof}

Just out of curiosity, let us also state what the same proof shows for $K_{2}
$:

\begin{example}
If $F$ is a finite field, then $K_{2}(\mathsf{LCA}_{F})\cong K_{1}(F)\cong
F^{\times}$.
\end{example}

\begin{example}
If $F$ is a finite extension of a $p$-adic field, then $K_{2}(\mathsf{LCA}%
_{F,\operatorname*{top}})\cong\mathbf{\mu}(F)\oplus\mathbb{Q}^{I}$ for some
uncountable cardinal $I$, by Moore's Theorem identifying the $K_{2}$-group of
a local field with its group of units $\mathbf{\mu}$ and a uniquely divisible
part \cite[Chapter III, Theorem 6.2.4]{MR3076731}.
\end{example}

\begin{example}
If $F$ is a number field, $K_{2}(\mathsf{LCA}_{F})$ is isomorphic to the
cokernel of the diagonal embedding%
\[
\mathbf{\mu}(F)\overset{\iota}{\longrightarrow}\left\{  \left.  (\alpha
_{P})_{P}\in\prod_{P}\mathbf{\mu}(\widehat{F}_{P})\right\vert
\begin{array}
[c]{l}%
\alpha_{P}\text{ is prime-to-}\operatorname*{char}(\kappa(P))\text{ torsion
}\\
\text{for all but finitely many places}%
\end{array}
\right\}  \oplus\mathbb{Q}^{I}\text{,}%
\]
where $\kappa(P):=\mathcal{O}/P$ denotes the residue field at $P$, and $I$
some uncountable cardinal. This is shown by combining Theorem
\ref{thm_MainStructureThmOnKLCAFab}, Moore's Theorem to identify the torsion
of $K_{2}(\widehat{F}_{P})$ with all roots of unity of the local field, and
the localization sequence of $\operatorname*{Spec}\widehat{\mathcal{O}}%
_{P}/P\hookrightarrow\operatorname*{Spec}\widehat{\mathcal{O}}_{P}%
\hookleftarrow\operatorname*{Spec}\widehat{F}_{P}$ along with the observation
that it is splits \cite[Chapter V, Corollary 6.9.2]{MR3076731}, and the image
of the splitting are precisely the prime-to-$\operatorname*{char}(\kappa(P))$
torsion roots of unity. As the isomorphism of Moore's Theorem is the Hilbert
symbol, the finiteness condition in Example (3) is philosophically consistent
with ramification only occurring at finitely many places.
\end{example}

\begin{conjecture}
The morphisms $K_{n}(F)\rightarrow K_{n}(\mathsf{LCA}_{F,ab})$ are injective
for all $n\in\mathbb{Z}$.
\end{conjecture}

This conjecture is certainly true rationally, and more broadly the kernel can
at worst be a finite abelian group. This already follows from mapping only to
the infinite places and using the same kind of argument as in \cite[Section
5]{obloc}.%

\appendix

\section{Auxiliary computations}

\subsection{Locally compact modules over finite fields}

\begin{proposition}
\label{prop_KLCA_for_finite_fields}Let $\mathsf{A}$ be any stable $\infty
$-category and $K:\operatorname*{Cat}_{\infty}^{\operatorname*{ex}}%
\rightarrow\mathsf{A} $ be a localizing invariant (in the sense of
\cite{MR3070515}). Then there is an equivalence $K(\mathsf{LCA}_{\mathbb{F}%
_{q}})\overset{\sim}{\longrightarrow}\Sigma K(\mathbb{F}_{q})$.
\end{proposition}

\begin{proof}
\textit{(Step 1)} Let $G$ be in $\mathsf{LCA}_{\mathbb{F}_{q}}$. Pick a number
field $F$, denote its ring of integers by $\mathcal{O}$ and choose a prime $P$
such that $\mathcal{O}/P\cong\mathbb{F}_{q}$. Then via $\mathcal{O}%
\rightarrow\mathcal{O}/P\cong\mathbb{F}_{q}$ we may also regard $G$ as an
object in $\mathsf{LCA}_{\mathcal{O}}$ and Levin's structure theory applies,
\cite[Theorem 2 and Theorem 3]{MR0310125}. As $G$ is $p$-torsion (genuinely,
not just topologically), there cannot be a vector $\mathcal{O}$-module
contribution, i.e. there exists a compact clopen $\mathcal{O}$-submodule $C$
such that%
\begin{equation}
C\hookrightarrow G\twoheadrightarrow D\label{lw_TatelikeSequence}%
\end{equation}
in exact in $\mathsf{LCA}_{\mathcal{O}}$ with $D$ discrete. As $G$ is an
$\mathcal{O}$-module annihilated by $P$, so are $C$ and $D$. Thus, $D$ is a
discrete $\mathbb{F}_{q}$-vector space, and $C$ the Pontryagin dual of a
discrete $\mathbb{F}_{q}$-vector space. \textit{(Step 2)} There is an exact
functor%
\[
\Gamma:\mathsf{Tate}(\mathsf{Vect}_{fd}(\mathbb{F}_{q}))\longrightarrow
\mathsf{LCA}_{\mathbb{F}_{q}}\text{,}%
\]
sending a formal ind-pro limit of finite-dimensional $\mathbb{F}_{q}$-vector
spaces to its evaluation in $\mathsf{LCA}$. The construction is a mild
variation of the analogous functor $\gamma:\mathsf{Tate}(\mathsf{Mod}%
_{fin}(\mathcal{O}))\rightarrow\mathsf{LCA}_{\mathcal{O}}$ which was set up in
\cite[Section 6]{obloc}, so we shall not repeat it here. It suffices to
observe that $\mathsf{Vect}_{fd}(\mathbb{F}_{q})$ can be regarded as a fully
exact subcategory of $\mathsf{Mod}_{fin}(\mathcal{O})$ and that the
topological $\mathcal{O}$-module structure on $\gamma(G)$ is still annihilated
by $P$, so we may regard it as an $\mathcal{O}/P\cong\mathbb{F}_{q}$-vector
space structure. The functor $\Gamma$ is fully faithful. It is also
essentially surjective since by Step 1 every $G$ has a presentation as in
Equation \ref{lw_TatelikeSequence}, but this corresponds precisely to the
property to have a lattice in the sense of \cite{MR3510209}, see Definition
5.5 and Theorem 5.6 \textit{loc. cit.} Thus, the computation of the invariant
$K$ reduces to the corresponding computation for the Tate category. By Saito's
delooping theorem \cite{MR3317759}, we obtain%
\[
K(\mathsf{LCA}_{\mathbb{F}_{q}})\underset{\Gamma}{\overset{\sim}%
{\longrightarrow}}K(\mathsf{Tate}(\mathsf{Vect}_{fd}(\mathbb{F}_{q}%
)))\overset{\sim}{\longrightarrow}\Sigma K(\mathsf{Vect}_{fd}(\mathbb{F}%
_{q}))=\Sigma K(\mathbb{F}_{q})\text{,}%
\]
giving the claim. Saito's result is only stated for non-connective $K$-theory,
but the proof generalizes.
\end{proof}

\subsection{Locally compact modules over local fields}

In this section, exceptionally, let $F$ be a finite extension of the $p$-adics
$\mathbb{Q}_{p}$ for some prime number $p$.

\begin{definition}
Let $\mathsf{LCA}_{F,\operatorname*{top}}$ be the category of locally compact
topological $F$-modules. Unlike in the rest of this text, we demand that the
scalar multiplicaton%
\[
F\times M\longrightarrow M
\]
gives a topological $F$-module structure, where $F$ is equipped with its
valuation topology (rather than the discrete one!).
\end{definition}

\begin{lemma}
The category $\mathsf{LCA}_{F,\operatorname*{top}}$ is quasi-abelian, and in
particular naturally an exact category. If $\mathsf{LCA}_{F}$ denotes (as
usual in this paper) the category of locally compact $F$-modules, but where
$F$ is read with the discrete topology, then there is an exact functor%
\[
\mathsf{LCA}_{F,\operatorname*{top}}\longrightarrow\mathsf{LCA}_{F}\text{.}%
\]

\end{lemma}

\begin{proof}
Being quasi-abelian can be proven as by Hoffmann and Spitzweck \cite[Prop.
1.2]{MR2329311} for $\mathsf{LCA}$ plain, and the exact structure then stems
from \cite[Prop. 4.4]{MR2606234}.
\end{proof}

\begin{proposition}
\label{prop_LCA_F_overlocalfield}Suppose $F$ is a finite extension of the
$p$-adics $\mathbb{Q}_{p}$ for some prime number $p$. Then there is an exact
equivalence of exact categories%
\[
\mathsf{LCA}_{F,\operatorname*{top}}\overset{\sim}{\longrightarrow
}\mathsf{Vect}_{fd}(F)\text{.}%
\]

\end{proposition}

\begin{proof}
\textit{(Step 1)} Although a direct proof is possible, we will prove this by
reducing it to our formalism for number fields. Let $F_{0}$ be a number field
such that%
\[
F_{0}\otimes_{\mathbb{Q}}\mathbb{Q}_{p}\cong F\text{.}%
\]
(\textit{Proof:}\ Such a number field always exists. Let $f$ be a minimal
polynomial generating the extension $F/\mathbb{Q}_{p}$. Then $f\in
\mathbb{Q}_{p}[X]$ and we may pick a sequence of polynomials $f_{n}%
\in\mathbb{Q}[X]$ of the same degree with rational coefficients such that
$\lim_{n\rightarrow\infty}f_{n}=f$ in the $p$-adic topology on the
coefficients. Such a sequence exists as $\mathbb{Q}$ is dense in
$\mathbb{Q}_{p}$ and the space of coefficients is $\mathbb{Q}_{p}^{\deg f}$.
Even though $(f_{n})_{n}$ need not become stationary as a sequence of
polynomials, nor the sequence of number fields $T_{n}$ obtained by adjoining
all roots of $f_{n}$, the sequence $T_{n}\cdot\mathbb{Q}_{p}$ must become
stationary by\ Krasner's\ Lemma, and by convergence the limit must be $F$;
take $F_{0}$ to be any number field $T_{n}$ such that $T_{n}\cdot
\mathbb{Q}_{p}=F$). Next, pick an element $\pi\in F_{0}$ which has the
property that under $F_{0}\hookrightarrow F$ it becomes a uniformizer with
respect to the natural valuation on $F$. \textit{(Step 2)} Let $G\in
\mathsf{LCA}_{F,\operatorname*{top}}$ be an arbitrary object. Since for
$\mathsf{LCA}_{F,\operatorname*{top}}$ we had assumed the $F$-module structure
to be continuous with regards to the natural valuation topology on $F$
(instead of just the discrete topology), we have $\lim_{n\rightarrow\infty}%
\pi^{n}=0$ in $F$, and thus for all $g\in G$ we get%
\[
\lim_{n\rightarrow\infty}\pi^{n}\cdot g=0\text{.}%
\]
Since $\pi\in F_{0}$, it follows that $G$ is a topological $P$-torsion group,
for some prime $P$ of $\mathcal{O}_{F_{0}}$ (and more concretely, a prime such
that $P\cap\mathbb{Z}=(p)$, and $P\cdot F=(\pi)F$, i.e. after going to our
local field $F$, the prime becomes principal and is generated by the image of
$\pi$). By Proposition \ref{prop_TopTorsionIsVectorFreeAdelicBlock} we
conclude that $G$ is a vector-free adelic block. As $G$ was arbitrary, we
learn that all objects in $\mathsf{LCA}_{F,\operatorname*{top}}$ are
topological $P$-torsion modules for this concrete $P$. We will not repeat the
proof, but a mild variant of our version of the Braconnier--Vilenkin Theorem
(Theorem \ref{thm_StrongBracVilenkin}) then produces the claimed equivalence
of categories (The only necessary changes in the proof are to remove the
contributions from all infinite places and all other finite places except for
$P$, since the above argument rules out that any of these occur in
$\mathsf{LCA}_{F,\operatorname*{top}}$. The rest of the proof carries over verbatim).
\end{proof}

\begin{corollary}
\label{cor_Main_KThy_LocalCase}Let $\mathsf{A}$ be any stable $\infty
$-category and $K:\operatorname*{Cat}_{\infty}^{\operatorname*{ex}}%
\rightarrow\mathsf{A} $ be a localizing invariant (in the sense of
\cite{MR3070515}). Then there is an equivalence $K(\mathsf{LCA}%
_{F,\operatorname*{top}})\overset{\sim}{\longrightarrow}K(F)$.
\end{corollary}

\subsection{Grothendieck--Witt theory}

The following conjecture seems plausible:

\begin{conjecture}
\label{conj3}If $(\mathsf{C}_{i})_{i\in I}$ (for some index set $I$) are exact
categories with duality, then there is a canonical equivalence of spectra%
\[
GW\left(
{\textstyle\prod\nolimits_{i}}
\mathsf{C}_{i}\right)  \overset{\sim}{\longrightarrow}%
{\textstyle\prod\nolimits_{i}}
GW\left(  \mathsf{C}_{i}\right)  \text{.}%
\]

\end{conjecture}

This would be a natural analogue of \cite{MR1351941}.

\begin{theorem}
Let $F$ be a number field. If the Conjecture \ref{conj3} is true, then there
are canonical isomorphisms%
\[
GW_{n}(\mathsf{LCA}_{F,ab})\cong\left\{  \left.  (\alpha_{P})_{P}\in\prod
_{P}GW_{n}(\widehat{F}_{P})\right\vert
\begin{array}
[c]{l}%
\alpha_{P}\in GW_{n}(\widehat{\mathcal{O}}_{P})\text{ for all but finitely}\\
\text{many among the finite places}%
\end{array}
\right\}  \text{.}%
\]

\end{theorem}

\begin{proof}
The proof of Theorem \ref{thm_MainStructureThmOnKLCAFab} generalizes since all
underlying exact equivalences preserve duality.
\end{proof}

We should point out that our proof of Theorem \ref{thm_MainKOfLCAF} relies on
several subcategories which are not closed under duality, and in particular it
teaches us nothing about anything relying on duals.

\begin{problem}
What is $GW(\mathsf{LCA}_{F})$?
\end{problem}

\bibliographystyle{amsalpha}
\bibliography{ollinewbib}

\def\cprime{$'$} \def\polhk#1{\setbox0=\hbox{#1}{\ooalign{\hidewidth
  \lower1.5ex\hbox{`}\hidewidth\crcr\unhbox0}}} \def\cprime{$'$}
  \def\cprime{$'$} \def\cprime{$'$} \def\cprime{$'$}
\providecommand{\bysame}{\leavevmode\hbox to3em{\hrulefill}\thinspace}
\providecommand{\MR}{\relax\ifhmode\unskip\space\fi MR }
\providecommand{\MRhref}[2]{%
  \href{http://www.ams.org/mathscinet-getitem?mr=#1}{#2}
}
\providecommand{\href}[2]{#2}
\begin{thebibliography}{BGW16}

\bibitem[Arm81]{MR637201}
D.~Armacost, \emph{The structure of locally compact abelian groups}, Monographs
  and Textbooks in Pure and Applied Mathematics, vol.~68, Marcel Dekker, Inc.,
  New York, 1981. \MR{637201}

\bibitem[BGT13]{MR3070515}
A.~Blumberg, D.~Gepner, and G.~Tabuada, \emph{A universal characterization of
  higher algebraic {$K$}-theory}, Geom. Topol. \textbf{17} (2013), no.~2,
  733--838. \MR{3070515}

\bibitem[BGW16]{MR3510209}
O.~Braunling, M.~Groechenig, and J.~Wolfson, \emph{Tate objects in exact
  categories}, Mosc. Math. J. \textbf{16} (2016), no.~3, 433--504, With an
  appendix by Jan {\v{S}}{\v{t}}ov{\'{\i}}{\v{c}}ek and Jan Trlifaj.
  \MR{3510209}

\bibitem[Bra14]{MR3291352}
O.~Braunling, \emph{Geometric two-dimensional id\`eles with cycle module
  coefficients}, Math. Nachr. \textbf{287} (2014), no.~17-18, 1954--1971.
  \MR{3291352}

\bibitem[Bra17]{obloc}
\bysame, \emph{K theory of locally compact modules over rings of integers},
  arXiv:1710.10819 [math.KT] (2017).

\bibitem[B{\"u}h10]{MR2606234}
T.~B{\"u}hler, \emph{Exact categories}, Expo. Math. \textbf{28} (2010), no.~1,
  1--69. \MR{2606234 (2011e:18020)}

\bibitem[Car95]{MR1351941}
G.~Carlsson, \emph{On the algebraic {$K$}-theory of infinite product
  categories}, $K$-Theory \textbf{9} (1995), no.~4, 305--322. \MR{1351941}

\bibitem[Cla17]{clausen}
D.~Clausen, \emph{A {K}-theoretic approach to {A}rtin maps}, arXiv:1703.07842
  [math.KT] (2017).

\bibitem[HS07]{MR2329311}
N.~Hoffmann and M.~Spitzweck, \emph{Homological algebra with locally compact
  abelian groups}, Adv. Math. \textbf{212} (2007), no.~2, 504--524. \MR{2329311
  (2009d:22006)}

\bibitem[Kel96]{MR1421815}
B.~Keller, \emph{Derived categories and their uses}, Handbook of algebra,
  {V}ol.\ 1, North-Holland, Amsterdam, 1996, pp.~671--701. \MR{1421815
  (98h:18013)}

\bibitem[Kry97]{MR1620000}
N.~I. Kryuchkov, \emph{Injective and projective objects in the category of
  locally compact modules over the ring of integers of a global field}, Mat.
  Zametki \textbf{62} (1997), no.~1, 118--123. \MR{1620000}

\bibitem[KW17]{kaspwinges}
D.~Kasprowski and C.~Winges, \emph{Shortening binary complexes and
  commutativity of $k$-theory with infinite products}, arXiv:1705.09116
  [math.KT] (2017).

\bibitem[Lev73]{MR0310125}
M.~Levin, \emph{Locally compact modules}, J. Algebra \textbf{24} (1973),
  25--55. \MR{0310125}

\bibitem[Mor77]{MR0442141}
S.~Morris, \emph{Pontryagin duality and the structure of locally compact
  abelian groups}, Cambridge University Press, Cambridge-New York-Melbourne,
  1977, London Mathematical Society Lecture Note Series, No. 29. \MR{0442141}

\bibitem[Mos67]{MR0215016}
M.~Moskowitz, \emph{Homological algebra in locally compact abelian groups},
  Trans. Amer. Math. Soc. \textbf{127} (1967), 361--404. \MR{0215016}

\bibitem[Rob67]{MR0217211}
L.~C. Robertson, \emph{Connectivity, divisibility, and torsion}, Trans. Amer.
  Math. Soc. \textbf{128} (1967), 482--505. \MR{0217211}

\bibitem[Sai15]{MR3317759}
S.~Saito, \emph{On {P}revidi's delooping conjecture for {$K$}-theory}, Algebra
  Number Theory \textbf{9} (2015), no.~1, 1--11. \MR{3317759}

\bibitem[Sch04]{MR2079996}
M.~Schlichting, \emph{Delooping the {$K$}-theory of exact categories}, Topology
  \textbf{43} (2004), no.~5, 1089--1103. \MR{2079996 (2005k:18023)}

\bibitem[Sch10]{MR2600285}
\bysame, \emph{Hermitian {$K$}-theory of exact categories}, J. K-Theory
  \textbf{5} (2010), no.~1, 105--165. \MR{2600285}

\bibitem[Wei13]{MR3076731}
C.~Weibel, \emph{The {$K$}-book}, Graduate Studies in Mathematics, vol. 145,
  American Mathematical Society, Providence, RI, 2013, An introduction to
  algebraic $K$-theory. \MR{3076731}

\bibitem[Wu92]{MR1173767}
Sheng~L. Wu, \emph{Classification of self-dual torsion-free {LCA} groups},
  Fund. Math. \textbf{140} (1992), no.~3, 255--278. \MR{1173767}

\end{thebibliography}

\end{document}